\documentclass[11pt,a4paper]{article}
\usepackage{bbm}

\usepackage[leqno]{amsmath}
\usepackage{dsfont}
\usepackage{amsfonts}
\usepackage{graphicx}
\usepackage{amsmath}
\usepackage{amssymb}
\usepackage{latexsym}
\usepackage{amsmath, amsfonts,amssymb, amsthm, euscript,makeidx,color,mathrsfs}
\usepackage{enumerate}

\usepackage[colorlinks,linkcolor=blue,anchorcolor=green,citecolor=red]{hyperref}%when print, use the next package
\def\sqr#1#2{{\vcenter{\vbox{\hrule height.#2pt
              \hbox{\vrule width.#2pt height#1pt \kern#1pt \vrule width.#2pt}
              \hrule height.#2pt}}}}
\def\5n{\negthinspace \negthinspace \negthinspace \negthinspace \negthinspace }
\def\4n{\negthinspace \negthinspace \negthinspace \negthinspace }
\def\3n{\negthinspace \negthinspace \negthinspace }
\def\2n{\negthinspace \negthinspace }
\def\1n{\negthinspace }

\def\dbD{\mathbb{D}}
\def\dbE{\mathbb{E}}
\def\dbF{\mathbb{F}}

\def\dbH{\mathbb{H}}

\def\dbK{\mathbb{K}}
\def\dbL{\mathbb{L}}

\def\dbN{\mathbb{N}}

\def\dbR{\mathbb{R}}

\def\dbU{\mathbb{U}}
\def\dbV{\mathbb{V}}

\def\dbX{\mathbb{X}}

\def\dbZ{\mathbb{Z}}

\def\={\buildrel \triangle \over =}

\def\ds{\displaystyle}

\def\nb{\noalign{\bs}}

\def\a{\alpha}

\def\d{\delta}

\def\si{\sigma}
\def\t{\tau}

\def\th{\theta}

\def\G{\Gamma}
\def\D{\Delta}

\def\O{\Omega}
\def\mf{\mathcal{F}}
\def\me{\mathbb{E}}
\def\rd{\,\mathrm d}  
\def\bal{\begin{aligned}}
\def\eal{\end{aligned}}
\def\nb{\nabla}
\def\cF{{\cal F}}

\def\cJ{{\cal J}}

\def\cL{{\cal L}}

\def\cP{{\cal P}}

\def\cS{{\cal S}}
\def\cT{{\cal T}}

\def\cl{{\cal l}}
\def\no{\noindent}

\def\ms{\medskip}

\def\q{\quad}
\def\qq{\qquad}

\def\lt{\left}
\def\rt{\right}

\def\h{\widehat}
\def\wt{\widetilde}

\def\cd{\cdot}
\def\cds{\cdots}

\def\cl{\overline}

\def\({\Big (}
\def\){\Big )}
\def\[{\Big[}
\def\]{\Big]}

\def\bde{\begin{definition}\label}
\def\ede{\end{definition}}
\def\be{\begin{equation}}
\def\bel{\begin{equation}\label}
\def\ee{\end{equation}}
\def\beq{\begin{equation*}\begin{aligned}}
\def\eeq{\end{aligned}\end{equation*}}
\def\bt{\begin{theorem}\label}
\def\et{\end{theorem}}
\def\bc{\begin{corollary}\label}
\def\ec{\end{corollary}}
\def\bl{\begin{lemma}\label}
\def\el{\end{lemma}}
\def\bp{\begin{proposition}\label}
\def\ep{\end{proposition}}
\def\bas{\begin{assumption}\label}
\def\eas{\end{assumption}}
\def\br{\begin{remark}\label}
\def\er{\end{remark}}
\def\bex{\begin{example}\label}
\def\ex{\end{example}}
\def\ba{\begin{array}}
\def\ea{\end{array}}
\def\ed{\end{document}}
\def\square#1{\vbox{\hrule\hbox{\vrule height#1%
     \kern#1\vrule}\hrule}}
\def\rectangle#1#2{\vbox{\hrule\hbox{\vrule height#1%
     \kern#2\vrule}\hrule}}

\font\tenbb=msbm10 \font\sevenbb=msbm7 \font\fivebb=msbm5

\newfam\bbfam
\scriptscriptfont\bbfam=\fivebb \textfont\bbfam=\tenbb
\scriptfont\bbfam=\sevenbb

\newtheorem{theorem}{\hskip 1.3em Theorem}[section]
\newtheorem{definition}[theorem]{\hskip 1.3em Definition}
\newtheorem{proposition}[theorem]{\hskip 1.3em Proposition}
\newtheorem{corollary}[theorem]{\hskip 1.3em Corollary}
\newtheorem{lemma}[theorem]{\hskip 1.3em Lemma}
\newtheorem{remark}[theorem]{\hskip 1.3em Remark}
\newtheorem{example}[theorem]{\hskip 1.3em Example}
\newtheorem{algorithm}[theorem]{\hskip 1.3em Algorithm}

\newtheorem{assumption}[theorem]{\hskip 1.3em Assumption}
\makeatletter
   
   \@addtoreset{equation}{section}
\makeatother

\begin{document}

\title{Strong Rates of Convergence for Space-Time Discretization of the
Backward Stochastic Heat Equation, and of a
Linear-Quadratic Control Problem for the Stochastic Heat Equation\thanks{This work is supported in part by
the National Natural Science Foundation of China (11801467), and the Chongqing Natural Science Foundation
(cstc2018jcyjAX0148).}}

\author{Andreas Prohl\thanks{
Mathematisches Institut, Universit\"at T\"ubingen, Auf der Morgenstelle 10,
D-72076 T\"ubingen, Germany.
 {\small\it
e-mail:} {\small\tt prohl@na.uni-tuebingen.de}.} \quad 
and \quad Yanqing Wang\thanks{Corresponding author.
School of Mathematics and Statistics, Southwest University, Chongqing 400715, China.  {\small\it
e-mail:} {\small\tt yqwang@amss.ac.cn}. \ms}}

\date{January 18, 2020}
\maketitle

\begin{abstract}
We introduce a time-implicit, finite-element based space-time discretization scheme for the backward stochastic heat equation, 
and for the forward-backward stochastic heat equation from
stochastic optimal control, and prove  strong rates of convergence. The
fully discrete version of the forward-backward stochastic heat equation is then used within a gradient descent algorithm  
to approximately solve the linear-quadratic control problem for the stochastic heat equation driven 
by  additive noise.
\end{abstract}

\ms

\no\bf Keywords: \rm Strong error estimate, backward stochastic heat equation, stochastic linear quadratic problem, 
%stochastic heat equation, 

%finite element method, 

\ms

\no\bf AMS 2010 subject classification: \rm 49J20,
%{\color{red}(Optimal control problems involving partial differential equations)},
 65M60,
 %{\color{red}(Finite elements, Rayleigh-Ritz and Galerkin methods, finite methods)}, 
 93E20
 %{\color{red}(Optimal stochastic control)}

\section{Introduction}

Let $D \subset {\mathbb R}^d$ be a bounded domain with $C^2$ boundary, $T>0$, and a  (deterministic) 
function $\widetilde{X} \equiv \{ \widetilde{X}(t);\,  t \in  [0,T]\}\in C([0,T];\dbH_0^1 \cap {\mathbb H}^2)$ be given. 
Our goal is to numerically approximate the ${\mathbb L}^2$-valued, ${\mathbb F}$-adapted control process 
$U^* \equiv \{ U^*(t);\, t \in [0,T]\}$ on the filtered probability space 
$(\Omega, {\mathcal F}, {\mathbb F}, {\mathbb P})$ that minimizes the functional ($\alpha \geq 0$)
\begin{equation} \label{w1003e2}
{\mathcal J}(X, U)=\frac 1 2 {\mathbb E} \Bigl[\int_0^T\lt( \Vert X(t) - \widetilde{X}(t) \Vert_{{\mathbb L}^2}^2+ \Vert U(t) \Vert^2_{{\mathbb L}^2}\rt)\, {\rm d}t
+ \alpha \Vert X(T) - \widetilde{X}(T) \Vert^2_{{\mathbb L}^2}\Bigr] 
\end{equation}
subject to the (controlled forward) stochastic heat equation ({\bf SPDE} for short)
\bel{w1013e1}
\lt\{
\begin{array}{ll}
{\rm d}X(t)=\bigl[\D X(t)+U(t)\bigr]\, {\rm d}t+\si(t) {\rm d}W(t) \q &\forall\, t \in [0,T]\,,\\
%X(t,x)=0,\q\q & (t,x)\in (0,T)\times \partial D,\\
X(0)=X_0\,,
\end{array}
\rt.
\ee
which is supplemented by homogeneous Dirichlet boundary condition. Here
$W \equiv \{ W(t);\, t \in [0,T]\}$ is an ${\mathbb R}^m$-valued Wiener process, and $\sigma \equiv \{ \sigma(t);\, t \in [0,T]\} \in L^2_{\mathbb F}\bigl(\Omega; L^2(0,T; \dbH_0^1\cap \dbH^2)\bigr)$, which both are given on the same filtered probability space. For every given $U \in L^2_{\mathbb F}\bigl(\Omega; L^2(0,T; {\mathbb L}^2)\bigr)$, there exists a unique ${\mathbb H}^1_0$-valued strong (variational) solution $X \equiv X(U)$ in  \eqref{w1013e1} such that
$X(0) = X_0 \in L^2(\Omega; \dbH^1_0)$, and a 
unique minimizer $(X^*, U^*) \in L^2_{{\mathbb F}}\bigl( \Omega; C([0,T]; {\mathbb L}^2) \cap L^2(0,T; {\mathbb H}^1_0) \times L^2(0,T; {\mathbb L}^2)\bigr)$ of the stochastic optimal control problem: 'minimize \eqref{w1003e2} subject to \eqref{w1013e1}' -- which we below refer to as {\bf SLQ}; see e.g.~\cite{Bensoussan83}.

We consider problem {\bf SLQ} as a prototype example of a (linear-quadratic) stochastic optimal control problem involving a stochastic PDE, for which  corresponding numerical analyses so far are rare in the existing literature; see e.g.~\cite{Wang16,Dunst-Prohl16}. This is in contrast to the deterministic counterpart problem {\bf LQ} which involves a linear PDE,
where optimal rates of {convergence} are available for (finite element based) space-time discretization of related optimality conditions (see e.g.~\cite{Mcknight-Bosarge73,Malanowski82,Rosch04,Meidner-Vexler08,Gong-Hinze13}), which may then be used as part of a gradient descent algorithm with step size control 
\cite{Hinze-Pinnau-Ulbrich-Ulbrich09}
%\footnote{I suggest to here cite the book by Hinze et al.} 
to approximate
the minimizing tuple $(X^*, U^*)$, which here consists of deterministic state and control functions. If compared to problem {\bf LQ}, problem {\bf SLQ} owns some distinctive characters and additional difficulties caused by the driving
Wiener process in the SPDE \eqref{w1013e1}, which make the generalization of the
numerical results for the deterministic control problem to {\bf SLQ} a non-trivial task.
For example, a crucial difficulty consists in solving the adjoint equation in the context of {\bf SLQ}, which here is  a backward stochastic PDE ({\bf BSPDE} for short) of the form
\bel{bshe1}
\lt\{
\begin{array}{ll}
{\rm d}Y(t)= \Bigl[-\D Y(t)+ \bigl[X(t) - \widetilde{X}(t)\bigr] \Bigr]{\rm d}t+Z(t) {\rm d}W(t)  \q & \forall\, t \in [0,T]\, ,\\
%Y (t)=0    & \mbox{on }[0,T]\, ,\\
Y(T)=- \alpha \lt(X(T)-\wt X(T)\rt)\, ,
\end{array}
\rt.
\ee
having a solution tuple $(Y,Z) \in L^2_{{\mathbb F}}\bigl( \Omega;  C([0,T]; {\mathbb H}^1_0) \cap L^2(0,T; {\mathbb H}^1_0 \cap {\mathbb H}^2)\bigr) \times
L^2_{{\mathbb F}}\bigl( \Omega; L^2(0,T; {\mathbb H}^1_0)\bigr)$; cf.~\cite{Du-Tang12}.
%\footnote{K. Du and S. Tang, {\em Strong solution of backward stochastic partial differential equations in $C^2$ domains}, Probab. Theory Rel. 154 (2012), pp. 255--285.}. 
The adjoint variable $Y$ is then related to the optimal control
by Pontryagin's maximum principle, which in the case of problem {\bf SLQ} is
\begin{equation}\label{pontr1} 0 = U^*(t) - Y(t) \qquad \forall\, t \in (0,T)\, .
\end{equation}
The combination of equations \eqref{w1013e1}, \eqref{bshe1}, and \eqref{pontr1} then uniquely determines the optimal process tuple $(X^*,U^*)$ of problem {\bf SLQ}.

The convergence analysis of space-time discretization
{\bf BSPDE}s is a recent research subject, and available
results  are rare as well.
A first error analysis for an (abstract) time-space discretization based on the implicit Euler method for the 
above {\bf BSPDE} \eqref{bshe1} is \cite{Wang16}, where the
error depends on the ratio of time discretization and Galerkin parameters. That work
heavily draws conclusions with the help of Malliavin calculus, and
the (space-)time regularity of solutions to the underlying {\bf BS(P)DE}. In  \cite{Dunst-Prohl16}, 
the authors derive rates of convergence for a conforming finite element semi-discretization, and discuss its actual implementation.
The proofs in \cite{Dunst-Prohl16} use simple variational arguments, resting on improved regularity properties of the variational solution, It\^{o}'s formula, and approximation results for the finite element method. However, the 
interplay of spatial and temporal discretization errors is left open in \cite{Dunst-Prohl16}, in particular the relevant question regarding unconditional convergence rates, which allow discretization parameters w.r.t.~time and space to independently
tend to zero, and general regular space-time meshes.

We address this issue in Section \ref{error_bspde} as our first goal in this work: its derivation requires to study the
time regularity of the solution $(Y,Z)$ to {\bf BSPDE} \eqref{bshe} in particular,  which seems not available in the literature so far, and uses Malliavin calculus for that matter for the solution component $Z$, in particular. For this purpose, we borrow related arguments from
\cite{ZhangJF04,Hu-Nualart-Song11} where BSDEs are studied, and variational arguments.

The second goal in this work is addressed in Section \ref{rate}, where strong error estimates for a space-time discretization of the coupled forward-backward SPDE
\eqref{w1013e1}-\eqref{bshe1} ({\bf FBSPDE} for short) are shown. These results extend available ones 
(cf.~\cite{Dunst-Prohl16}) in the literature in several aspects: the obtained strong convergence rates for the used 
finite element based {\em space-time} discretization \eqref{w1212e3} holds for {\em arbitrary} times $T$ -- and is 
not only a semi-discretization in space where
optimal rates are obtained in \cite{Dunst-Prohl16} for {\em small times $T$} via a contraction argument. The numerical analysis uses variational arguments to first bound the error in the optimal controls, and here exploits the unique solvability of 
(the discretization of) problem {\bf SLQ}, as well as the related sufficient and necessary optimality conditions; in a second step, error bounds for the optimal state, and its adjoint are based on stability properties of the state equation, and the adjoint equation available from Section \ref{error_bspde}.

To solve a {\bf BSPDE} computationally requires huge computational resources (see \cite{Dunst-Prohl16}), and it is even more computationally demanding (in terms of computational storage requirements and computational times) to solve the coupled  {\bf FBSPDE}. 
Consequently, an alternative numerical strategy to
the space-time discretization of {\bf FBSPDE} is useful to make accurate computations feasible, which decouples the computation of (approximating iterates of) the solution parts from a {\bf SPDE} from that of a {\bf BSPDE} per iteration.
A simple fixed-point method on the level of optimality conditions to accomplish this goal is known to converge only for small times $T>0$ (cf.~\cite{Dunst-Prohl16}); instead,
we may again return to the fully discretized problem {\bf SLQ}$_{h\t}$
\eqref{w1003e8}--\eqref{w1003e7}
 and exploit its character as a minimization problem to initiate a gradient descent method to
successively determine approximations of the optimal control; this method is detailed in
Section \ref{numopt}, and a convergence order is shown for this iteration which is the final goal in this work.

The rest of this paper is organized as follows. In Section \ref{pre}, we introduce  notations, and review relevant properties of the problems {\bf BSPDE} \eqref{bshe} and {\bf FBSPDE} considered in this work. In Section \ref{error_bspde}, we prove  strong error estimates
for a space-time discretization of {\bf BSPDE}. By virtue of the obtained
 error  estimates, in Section \ref{rate}, we prove a convergence rate for a space-time discretization of {\bf FBSPDE}, which is related to problem {\bf SLQ}. Convergence of the related iterative gradient descent method towards the minimizer $U^*$ of  {\bf SLQ} is shown in Section \ref{numopt}.

\section{Preliminaries}\label{pre}

\subsection{Notation --- involved processes and the finite element method}\label{not1}

Let $\bigl({\mathbb K}, ( \cdot, \cdot)_{{\mathbb K}}\bigr)$ be a separable Hilbert space. 
By $\Vert \cdot \Vert_{{\mathbb L}^2}$ resp.~$(\cdot, \cdot)_{\dbL^2}$, 
we denote the norm resp.~the scalar product in ${\mathbb L}^2 := {\mathbb L}^2(D)$.
The norm in ${\mathbb H}^1_0:=H_0^1(D)$, $\dbH^2:=H^2(D)$ is denoted by $\| \cdot \|_{\dbH_0^1}$, $\|\cd\|_{\dbH^2}$
respectively. Let $(\Omega, {\mathcal F}, {\mathbb F}, {\mathbb P})$ be a complete filtered probability space, where 
${\mathbb F}$ is the filtration generated by the ${\mathbb R}^m$-valued Wiener process $W$, which is augmented by 
all the ${\mathbb P}$-null sets. Below, we set $m=1$ for simplicity. The space of all ${\mathbb F}$-adapted processes 
$X: \Omega \times [0,T] \rightarrow {\mathbb K}$ satisfying 
${\mathbb E}[\int_0^T \Vert X(t)\Vert^2_{{\mathbb K}}\, {\rm d}t] < \infty$ is denoted by 
$L^2_\dbF(\O;L^2(0,T; {\mathbb K}))$; the space of all ${\mathbb F}$-adapted processes 
$X: \Omega \times [0,T] \rightarrow {\mathbb K}$ satisfying 
${\mathbb E}[\sup_{t \in [0,T]} \Vert X(t)\Vert^2_{\mathbb K}] < \infty$ is denoted by
$L^2_{{\mathbb F}}\bigl(\Omega; C([0,T]; {\mathbb K})\bigr)$.

We partition the bounded domain $D \subset {\mathbb R}^d$ via a regular triangulation ${\mathcal T}_h$ into elements $K$ with maximum mesh size
$h := \max \{ {\rm diam}(K):\, K \in {\mathcal T}_h\}$, and consider spaces
$${\mathbb V}_h^1 :=  \{\phi \in {\mathbb H}^1_0:\, \phi \bigl\vert_K \in {\mathbb P}_1(K) \quad \forall\, K \in {\mathcal T}_h \}\,, \qquad 
{\mathbb V}_h^0 := \{\phi \in {\mathbb L}^2:\, \phi \bigl\vert_K \in {\mathbb P}_0(K) \quad \forall\, K \in {\mathcal T}_h \}\, ,$$
where ${\mathbb P}_i(K)$ denotes the space of {polynomials} of degree $i \,(i=0,1)$.
{The ${\mathbb L}^2$-projection $\Pi_h^i: {\mathbb L}^2 \rightarrow {\mathbb V}^i_h$ is defined by $(\Pi^{i}_h \xi - \xi, \phi_h)_{{\mathbb L}^2}= 0$ for all $\phi_h \in {\mathbb V}_h^i$.
%For simplicity, we denote $\Pi_h \equiv \Pi_h^{(1)}$}. 
We define the discrete Laplacean $\Delta_h: {\mathbb V}_h^1 \rightarrow {\mathbb V}_h^1$ by $(-\Delta_h \xi_h, \phi_h)_{{\mathbb L}^2} = (\nabla \xi_h, \nabla \phi_h)_{{\mathbb L}^2}$ for all $\xi_h, \phi_h \in {\mathbb V}_h^1$. 

We use approximation estimates for the projection $\Pi_h^1$, and an inverse estimate (cf.~\cite{Brenner-Scott08}) to conclude that
\begin{equation}\label{estim}\Vert \Delta_h \Pi_h^1 \xi\Vert_{{\mathbb L}^2} \leq C \Vert \nabla^2 \xi\Vert_{{\mathbb L}^2} \qquad \forall\, \xi \in {\mathbb H}^1_0 \cap {\mathbb H}^2\, ,
\end{equation} since
\begin{eqnarray*}
\Vert \Delta_h \Pi_h^1 \xi\Vert_{{\mathbb L}^2}^2 &=& -\bigl( \nabla [\Pi_h^1\xi - \xi], \nabla \Delta_h \Pi_h^1 \xi\bigr)_{{\mathbb L}^2} - (\nabla \xi, \nabla \Delta_h \Pi_h^1 \xi)_{{\mathbb L}^2}\\
&\leq& Ch \Vert \nabla^2 \xi\Vert_{{\mathbb L}^2} \Vert \nabla \Delta_h \Pi_h^1 \xi \Vert_{{\mathbb L}^2} + (\Delta \xi, \Delta_h \Pi_h^1 \xi)_{{\mathbb L}^2}\\
&\leq& C \bigl(\Vert \nabla^2 \xi\Vert_{{\mathbb L}^2}  + \Vert \Delta \xi\Vert_{{\mathbb L}^2} \bigr)\Vert \Delta_h \Pi_h^1 \xi\Vert_{{\mathbb L}^2}\, .
\end{eqnarray*}

\ms

We denote by $I_\tau = \{ t_{n}\}_{n=0}^N \subset [0,T]$ a time mesh with maximum step size $\tau :=\max\{t_{n+1}-t_n:\,n=0,1,\cds,N-1\}$, and
$\D_nW=W(t_n)-W(t_{n-1})$ for all $n=1,\cds,N$. 
For simplicity, we choose a uniform partition, i.e. $\t=T/N$ and $\t\leq 1$. 
The results in this work still hold for quasi-uniform partitions. 
%

%%%
\subsection{The stochastic heat equation ---  strong convergence rates for a space-time discretization}\label{opt-1}

For a given $U\in L^2_{\mathbb F} \bigl(\Omega; L^2(0,T; \dbL^2) \bigr)$,  and 
$X_0 \in \dbH^1_0$ in \eqref{w1013e1}, there exist a strong solution 
$X \in L^2_{{\mathbb F}}\bigl( \Omega; C([0,T]; {\mathbb H}^1_0) \cap L^2(0,T; {\mathbb H}^1_0 \cap {\mathbb H}^2)\bigr)$ 
solving ${\mathbb P}$-a.s.~for all $t \in [0,T]$
\begin{equation}\label{forw1}
\bal
&\bigl( X(t), \phi\bigr)_{\dbL^2} - (X_0, \phi)_{\dbL^2} + \int_0^t \lt[\bigl(\nabla X(s), \nabla \phi\bigr)_{\dbL^2} - \bigl(U(s), \phi\bigr)_{\dbL^2}\rt]\, {\rm d}s\\
&\qq\qq= \int_0^t \bigl( \sigma(s), \phi \bigr)_{\dbL^2}\, {\rm d}W(s)
\qquad \forall\, \phi \in {\mathbb H}^1_0\,,
\eal
\end{equation}
and a constant $C \equiv C(D,T)>0$ such that
\begin{equation}\label{stochheat1} 
{\mathbb E}\lt[ \sup_{t \in [0,T]} \Vert X(t)\Vert_{{\mathbb H}_0^1}^2 +
\int_0^T \Vert X(t)\Vert^2_{{\mathbb H}^2}\, {\rm d}t\rt] \leq C {\mathbb E}\lt[\Vert X_0\Vert^2_{{\mathbb H}_0^1} + \int_0^T\Vert U(t)\Vert^2_{{\mathbb L}^2}\, {\rm d}t
\rt]\, .
\end{equation}
A finite element discretization of \eqref{forw1} then reads: For all $t \in [0,T]$, find
$X_h \in L^2_{{\mathbb F}}\bigl(\Omega; C([0,T]; {\mathbb V}_h^1)\bigr)$ such that ${\mathbb P}$-a.s.~and for all times $t \in [0,T]$
\bel{spdedisch1}
\bal
&\bigl(X_h(t), \phi_h\bigr)_{\dbL^2} - (X_0, \phi_h)_{\dbL^2} + \int_0^t \bigl(\nabla X_h(s), \nabla \phi_h\bigr)_{\dbL^2} - \bigl( U(s), \phi_h\bigr)_{\dbL^2}\, {\rm d}s \\ 
&\qquad \qq = \int_0^t \bigl( \sigma(s), \phi_h\bigr)_{\dbL^2}\, {\rm d}W(s) \qquad \forall\, \phi_h \in {\mathbb V}_h^1 \, .
\eal
\ee
%\begin{eqnarray}\nonumber
%&&\bigl(X_h(t), \phi_h\bigr) - (X_0, \phi_h) + \int_0^t \bigl(\nabla X_h(s), \nabla \phi_h\bigr) - \bigl( U(s), \phi_h\bigr)\, {\rm d}s \\ \label{spdedisch1}
%&&\qquad = \int_0^t \bigl( \sigma(s), \phi_h\bigr)\, {\rm d}W(s) \qquad \forall\, \phi_h \in {\mathbb V}_h^1\, .
%\end{eqnarray}
Equation \eqref{spdedisch1} may be recast into the following SDE system,
\bel{sde}
\lt\{
\bal
&dX_h(t)=\bigl[\Delta_hX_h(t)+\Pi_h^1U(t)\bigr] {\rm d}t+ \Pi_h^1\si(t) {\rm d}W(t)\,, \\
&X_h(0)= \Pi_h^1X_{0}\, .
\eal
\rt.
\ee

The derivation of a strong error estimate is standard, and uses the improved (spatial) regularity properties of the strong variational solution,
\begin{equation}\label{esti-space1}
\sup_{t \in [0,T]} {\mathbb E}\bigl[ \Vert X_h(t) - X(t)\Vert^2_{{\mathbb L}^2}\bigr]
+ {\mathbb E}\Bigl[ \int_0^T \bigl\Vert \nabla \bigl[ X_h(t) - X(t)\bigr] \bigr\Vert^2_{{\mathbb L}^2}\, {\rm d}t\Bigr]
\leq Ch^2\, .
\end{equation}
We now consider a time-implicit discretization of \eqref{spdedisch1} on a partition $I_\tau$ of $[0,T]$. The problem then reads: For every $0 \leq n \leq N-1$, find a solution $X^{n+1}_h \in L^2_{{\mathcal F}_{t_{n+1}}}(\Omega; {\mathbb V}_h^1)$ such that ${\mathbb P}$-a.s.
\begin{equation}\label{esti-time1} 
(X^{n+1}_h - X^{n}_h, \phi_h)_{\dbL^2} + \tau \Bigl[(\nabla X_h^{n+1},\nabla \phi_h)_{\dbL^2} - 
(U(t_n), \phi_h)_{\dbL^2}\Bigr] = \bigl( \sigma(t_n), \phi_h\bigr)_{\dbL^2} \Delta_{n+1} W\, ,
\end{equation}
where $\D_{n+1}W:=W(t_{n+1})-W(t_n)$.
The verification of the error estimate (see \cite{Yan05})
\begin{equation}\label{euler1}
\max_{0 \leq n \leq N} {\mathbb E}\bigl[ \Vert X_h(t_n) - X^n_h\Vert^2_{{\mathbb L}^2}\bigr] + \tau \sum_{n=1}^N  {\mathbb E}\Bigl[ \bigl\Vert \nabla \bigl[ X_h(t_n) - X^n_h\bigr] \bigr\Vert^2_{{\mathbb L}^2}\Bigr]\leq C\tau
\end{equation}
rests on stability properties of the implicit Euler, as well as the bound
\bel{w1208e2}
\sum_{n=0}^{N-1}\int_{t_n}^{t_{n+1}}\me\bigl[\|X_h(t)-X_h(t_n)\|_{\dbH_0^1}^2\bigr]\rd t \leq C \t\, ,
%{\mathbb E}\Bigl[ \bigl\Vert \nabla \bigl[ X_h(t) - X_h(s)\bigr] \bigr\Vert^2_{{\mathbb L}^2}\Bigr] \leq C \vert t-s\vert \qquad \forall\, s,t \in [0,T]\,, 
\ee
which requires additional regularity properties of involved data, i.e.,
$X_0 \in  \dbH^1_0 \cap \dbH^2$, as well as 
$\si\in L^2_\dbF(\O;L^2(0,T;\dbH_0^1\cap \dbH^2))$, $U\in L^2_\dbF(\O;L^2(0,T;\dbH_0^1))$,
such that 
\beq
\sum_{n=0}^{N-1}\int_{t_n}^{t_{n+1}}\me \lt[\|U(t)-U(t_n)\|_{\dbL^2}^2+\|\si(t)-\si(t_n)\|_{\dbL^2}^2\rt]\rd t \leq C\t\, ,
\eeq
and the ${\mathbb H}^1$-stability of the ${\mathbb L}^2$-projection $\Pi_h^1$; cf.~\cite{Crouzeix-Thomee87,Bramble-Pasciak-Steinbach02}. 
%%%%%

\subsection{The backward stochastic heat equation --- a finite element based spatial discretization}\label{opt-2}
Let $Y_T \in L^2_{{\mathcal F}_T}\bigl( \Omega; {\mathbb H}^1_0\bigr)$ and $f \in L^2_{{\mathbb F}} \bigl(\Omega; L^2(0,T; {\mathbb L}^2)\bigr)$. A strong solution to the backward stochastic heat equation
\bel{bshe}
\lt\{
\begin{array}{ll}
{\rm d}Y(t)= \bigl[-\D Y(t)+f(t) \bigr] {\rm d}t+Z(t){\rm d}W(t)  & \forall\, t \in [0,T] \\
Y(T)=Y_T  
\end{array}
\rt.
\ee
is a pair of square integrable ${\mathbb F}$-adapted processes 
$$(Y,Z) \in L^2_{{\mathbb F}}\Bigl(\Omega; C([0,T]; {\mathbb H}^1_0) \cap L^2(0,T; {\mathbb H}^1_0 \cap {\mathbb H}^2\bigr) \times L^2(0,T; {\mathbb H}^1_0)\Bigr)$$
 such that ${\mathbb P}$-a.s.~for all times $t \in [0,T]$
 \begin{equation}\label{vari-1}
 \bal
&\bigl( Y_T, \phi\bigr)_{\dbL^2} - (Y(t),\phi)_{\dbL^2} - \int_t^T\lt[ \bigl(\nabla Y(s), \nabla \phi \bigr)_{\dbL^2} +  \bigl( f(s), \phi\bigr)_{\dbL^2}\rt]\, {\rm d}s \\
&\qq\qq= \int_t^T \bigl( Z(s), \phi\bigr)_{\dbL^2}\, {\rm d}W(s)
 \quad \forall\, \phi \in {\mathbb H}^1_0\,,
 \eal
 \end{equation}
 and there exists a constant $C \equiv C(D,T) > 0$ such that
 \begin{equation}\label{vari-2}
 {\mathbb E}\bigl[ \sup_{t \in [0,T]} \Vert Y(t)\Vert^2_{\dbH_0^1}\bigr]
 + {\mathbb E}\Bigl[ \int_0^T \Vert Y(t)\Vert^2_{{\mathbb H}^2} + 
 \Vert Z(t) \Vert^2_{\dbH_0^1}\, {\rm d}t\Bigr] \leq
 C {\mathbb E}\Bigl[ \Vert Y_T\Vert^2_{\dbH_0^1} + \int_0^T \Vert f(t) \Vert^2_{{\mathbb L}^2}\, {\rm d}t\Bigr]\, .
 \end{equation}
 The existence of a strong  solution to \eqref{bshe} in the above sense, as well as its uniqueness are shown in \cite{Du-Tang12}.
 
We now consider a finite element discretization  of the {\bf BSPDE} \eqref{bshe}. Let $Y_{T,h} \in L^2_{{\mathcal F}_T}(\Omega; {\mathbb V}_{h}^{1})$ 
be an approximation of $Y_T$. The problem {\bf BSPDE}$_h$ then reads: Find $(Y_h, Z_h) \in L^2_{{\mathbb F}}\bigl( \Omega; C([0,T]; {\mathbb V}_h^1)\bigr) \times L^2_{{\mathbb F}}\bigl( \Omega; L^2(0,T; {\mathbb V}_h^1)\bigr)$ such that ${\mathbb P}$-a.s.~for all $t \in [0,T]$ 
\begin{eqnarray}\label{vf2}
&&(Y_{T,h},\phi_h)_{\dbL^2} - \bigl( Y_h(t), \phi_h\bigr)_{\dbL^2}  - \int_t^T \bigl(\nabla Y_h(s), \nabla \phi_h \bigr)_{\dbL^2} +  \bigl( f(s), \phi_h\bigr)_{\dbL^2}\, {\rm d}s \\ \nonumber
 &&\qquad \qquad = \int_t^T \bigl( Z_h(s), \phi_h\bigr)_{\dbL^2}\, {\rm d}W(s)
 \qquad \forall\, \phi_h \in {\mathbb V}^1_h\,.
\end{eqnarray}
% \begin{eqnarray}\nonumber
%&& \bigl( Y_h(t), \phi_h\bigr) - (Y_{T,h},\phi_h) + \int_t^T \bigl(\nabla Y_h(s), \nabla \phi_h \bigr) +  \bigl( f(s), \phi_h\bigr)\, {\rm d}s \\ \label{vf2}
% &&\qquad = \sum_{i=1}^m\int_t^T \bigl( Z^i_h(s), \phi_h\bigr)\, {\rm d}W^i(s)
% \qquad \forall\, \phi_h \in {\mathbb V}^1_h\,.
% \end{eqnarray}
%
%
%
%\bel{vf2}
%\bal
%&(Y_h(t),\phi_h) - (Y_T,\phi_h) + \itn\\
%=&\int_t^T \lt(a(Y_h(s),\phi_h)+(f_h(s),\phi_h)\rt)ds+\int_t^T (Z_h(s),\phi_h)dW(s),\,t\in [0,T],\, \phi_h\in S_h.
%\eal
%\ee
Equation \eqref{vf2} is equivalent to the following system of BSDEs:
\bel{bsde}
\lt\{
\bal
&{\rm d}Y_h(t)=\bigl[-\Delta_hY_h(t)+ \Pi_h^1 f(t)\bigr] {\rm d}t+ Z_h(t){\rm d}W(t) \quad \forall\, t\in [0,T]\\
&Y_h(T)=Y_{T,h}\, .
\eal
\rt.
\ee
The existence and uniqueness of a solution tuple $(Y_h, Z_h)$ e.g.~follows from \cite[Theorem 2.1]{ElKaroui-Peng-Quenez97}. Moreover, there exists $C \equiv C( f,T) >0$ such that
\begin{equation}\label{bspde19} 
\bal
&\sup_{t \in [0,T]} {\mathbb E}\bigl[ \Vert \nabla Y_h(t)\Vert^2_{{\mathbb L}^2}\bigr] +
{\mathbb E}\Bigl[ \int_0^T \Vert \Delta_h Y_h(t)\Vert^2_{{\mathbb L}^2} +
\Vert \nabla Z_h(t)\Vert^2_{{\mathbb L}^2}\, {\rm d}t\Bigr] \\
&\qq\leq
C {\mathbb E}\bigl[ \lt\| \nabla Y_{T,h}\rt\|^2_{{\mathbb L}^2}+\int_0^T\|f(t)\|^2_{\dbL^2}\rd t\bigr]\, ;
\eal
\end{equation}
cf.~\cite[Lemma 3.1]{Dunst-Prohl16}. --- The  following result is taken from \cite[Theorem 3.2]{Dunst-Prohl16}, whose proof exploits  the bounds \eqref{vari-2}.
\begin{theorem}\label{w0911t1}
Let $Y_T \in L^2_{{\mathcal F}_T}(\Omega; {\mathbb H}^1_0)$, $Y_{T,h} \in L^2_{{\mathcal F}_T}(\Omega; {\mathbb V}^1_h)$. Let $(Y,Z)$ be the solution to \eqref{vari-1}, and $(Y_h, Z_h)$ solve \eqref{vf2}. There exists
$C \equiv C(Y_T, f,T) >0$ such that
\begin{eqnarray*}
&&\sup_{t \in [0,T]} \me \bigl[\|Y(t)-Y_h(t)\|^2_{{\mathbb L}^2}\bigr] +\me \Bigl[\int_0^T\|\nabla\bigl[ Y(t)-Y_h(t)\bigr]\|^2_{{\mathbb L}^2}+\|Z(t)-Z_h(t)\|^2_{{\mathbb L}^2}\, {\rm d}t\Bigr]\\
&&\qquad \leq C \bigl(\me\bigl[\|Y_T - Y_{T,h}\|^2_{{\mathbb L}^2} \bigr] +h^2\bigr)\, .
\end{eqnarray*}
\end{theorem}
Choosing $Y_{T,h} = \Pi_h^1 Y_T$ thus leads to an error estimate for the spatial semi-discretization \eqref{bsde}.
%%%%

\ms
\subsection{Temporal discretization of the backward stochastic heat equation --- the role of the Malliavin derivative}\label{temp-mall}

The numerical analysis of a temporal discretization of \eqref{bsde} requires Malliavin calculus to bound temporal increments $\me\Vert Z_h(t) - Z_h(s)\Vert_{{\mathbb L}^2}$ in terms of $\vert t-s \vert$, where $s,t \in [0,T]$. We therefore recall the definition of the Malliavin derivative
of processes, and the crucial connection between the Malliavin derivative of $Y_h$ and $Z_h$ from \eqref{bsde}. For further details, we refer to \cite{Nualart06,ElKaroui-Peng-Quenez97}. 

Let us recall that $\cF_{T}=\sigma\{W(t);0\leq t\leq T\}$, and that ${\mathbb K}$ denotes a separable Hilbert space. We define the It\^o isometry $W:L^2(0,T;{\mathbb R})\to L^2_{\cF_{T}}(\Omega;{\mathbb R})$ by
\beq
%\begin{eqnarray}
W(h)=\int_0^{T}h(t) \, {\mathrm d}W(t)\, .
%\end{eqnarray}
\eeq
For $\ell \in {\mathbb N}$, we denote by $C_p^\infty(\dbR^\ell)$ the {space} of all smooth functions $g:\dbR^\ell \to \dbR$ such that $g$ and all of its partial derivatives have  polynomial growth.
Let  $\cP$ be the {set} of ${\mathbb R}$-valued  random variables of the form
\begin{equation}\label{eq5.1}
F=g\bigl(W(h_1),\cdots,W(h_\ell) \bigr) 
\end{equation}
for some $g \in C_p^{\infty}(\dbR^\ell)$, $\ell \in {\mathbb N}$, and   $h_1,\ldots,h_\ell\in L^2(0,T;{\mathbb R})$.
To any $F\in \cP$ we define its  ${\mathbb R}$-valued Malliavin derivative 
$DF := \{ D_tF;\, 0 \leq t \leq T\}$ process via
\beq%\begin{eqnarray}
D_t F=\sum\limits_{i=1}^\ell \frac{\partial g}{\partial x_i} (W(h_1),\cdots,W(h_\ell))h_i(t)\, .
\eeq
In general, we can define the $k$-th iterated derivative of $F$ by $D^kF=D(D^{k-1}F)$, for any
$k\in {\mathbb N}$.

Now we extend the derivative operator to $\dbK$-valued variables. For any $k\in \dbN$, and $u$ in the set of $\dbK$-valued variables:
\beq
\cP_\dbK= \Bigl\{ u=\sum_{j=1}^n F_j \phi_j: F_j\in \cP,\,\phi_j\in \dbK,\, n\in \dbN  \Bigr\}\, ,
\eeq
we can define the $k$-th iterated derivative of $u$ by
\beq
D^k u=\sum_{j=1}^n D^kF_j\otimes \phi_j\, .
\eeq
For $p \geq 1$, we define the norm $\|\cdot\|_{k,p}$  via\beq%\begin{eqnarray}
\|u\|_{k,p} :=\bigg(\dbE \bigl[\|u\|_\dbK^{p} +\sum_{j=1}^k \lt\|D^ju\rt\|_{\lt({ L^2(0,T;{\mathbb R})}\rt)^{\otimes j}\otimes \dbK}^p\Bigr]\bigg)^{\frac 1 p}\, .
\eeq%\end{eqnarray}
Then $\dbD^{k,p}({\mathbb K})$ is the completion of $\cP$ under the norm $\|\cdot\|_{k,p}$.

We may now express $Z_h$ in BSDE \eqref{bsde} in terms of the Malliavin derivative of $Y_h$.
%
%{\color{blue}In Section \ref{error_bspde}, we will study the regularity of $Z_h$ of BSDE \eqref{bsde}, which can be 
%can be obtained in terms of the Malliavin derivative of $Y_h$.

\begin{lemma}[\cite{ElKaroui-Peng-Quenez97}, Prop.~5.3]\label{malliavin}
Suppose that  $Y_{T,h}\in \dbD^{1,2}(\dbL^2)$, $f \in L^2_{\dbF}\bigl(\Omega; L^2(0,T;\dbL^2)\bigr)$, and
\beq
\me \Bigl[\int_0^T\|D_\th Y_{T,h}\|^2_{\dbL^2}\, {\rm d}\th\Bigr] +\me\Bigl[\int_0^T\int_0^T\|D_\th f(t)\|^2_{\dbL^2}\, {\rm d}t {\rm d}\th\Bigr]<\infty\, .
\eeq 
Let $(Y_h ,Z_h )$ be the solution to BSDE \eqref{bsde}. Then 
$$(Y_h ,Z_h )\in L^2_\dbF\Bigl(\Omega; C\bigl([0,T];\dbD^{1,2}(\dbL^2)\bigr) \times L^2\bigl(0,T;\dbD^{1,2}(\dbL^2)\bigr)\Bigr)\,,$$ and its Malliavin derivative $(D_\th Y_h,D_\th Z_h)$ solves
\bel{w0911e20}
\lt\{
\bal
&D_\th Y_h(t) -D_\th Y_h(T)+\int_t^T - \Delta_hD_\th Y_h(s)+\Pi_h^1D_\th f(s) \, {\rm d}s\\
&\qq\qq\qq\qq\qq=-\int_t^T D_\th Z_h(s)\, {\rm d}W(s)    \qquad 0\leq \th\leq t\leq T \,, \\
&{D_\th Y_h(t)=D_\th Z_h(t)=0   \qquad 0\leq t<\th\leq T}\, .
\eal
\rt.
\ee
Moreover, $\{D_tY_h(t):\, 0\leq t\leq T\}$ is a version of $\{Z_h(t):\,0\leq t\leq T\}$.
\el
%
%{\color{red} In applications, the terminal datum $Y_{T,h}$ in BSDE (\ref{bsde}) comes from a (discretization of a) SPDE such as, e.g., 
%\beq%\bel{fshe}
%\lt\{
%\begin{array}{ll}
%{\rm d}X_h(t)= \bigl(\D_h X_h(t)+ \Pi_hf(t) \bigr){\rm d}t+\lt(X(t)+\si(t)\rt){\rm d}W(t), \q & \mbox{in } [0, T)\times D,\\
%X (t,x)=0,   & \mbox{on }[0,T)\times \pr D,\\
%X(0,x)=X_0(x),  & \mbox{in } D.
%\end{array}
%\rt.
%\eeq
%
%
% in particular, taking $Y_{T,h} = X_h(T)$ is motivated from (\ref{bshe1}), and will be relevant in Section \ref{rate}.} Then, by It\^{o}'s formula,
%$${\mathbb E}\bigl[\Vert X_h(t)\Vert_{{\mathbb L}^2}^2\bigr] \leq 
%C \Bigl( {\mathbb E}\bigl[\Vert X_0\Vert^2_{{\mathbb L}^2}\bigr] +
%\int_0^T {\mathbb E}\bigl[ \Vert U(\tau)\Vert^2_{{\mathbb L}^2} + \Vert \sigma(\tau)\Vert^2_{{\mathbb L}^2}\bigr]\, {\rm d}\tau\Bigr)\,. $$

%
%\subsection{The forward-backward stochastic heat equation}\label{opt-3}
%...

\section{Strong rates of convergence for a space-time discretization of the BSPDE \eqref{bshe}}\label{error_bspde}

In this section, we introduce the time discretization scheme \eqref{w0911e10} to approximate the solution $(Y_h, Z_h)$ to the {\bf BSPDE}$_h$ \eqref{bsde} by a finite sequence 
$\{\bigl( Y^n_h, Z^n_h\bigr) \}_{n=0}^{N-1}$ on a mesh $I_\tau$. The main results
are Theorems \ref{w0911t2} and \ref{w0911t2ab} in Subsection
\ref{error_bspde2}. Their derivation  crucially hinges on the time regularity of the solution $(Y_h, Z_h)$ to \eqref{bsde}, which is provided in the subsequent 
Subsection \ref{error_bspde_holder}.

\subsection{Uniform bounds for temporal increments of the solution $(Y_h, Z_h)$ to 
\eqref{bsde}}\label{error_bspde_holder}
We start with the derivation of uniform estimates for $Y_h$ which control its temporal increments.
We note again that all involved generic constants $C>0$ do not depend on $h$.

\begin{lemma}\label{w0911l2}
Suppose that $Y_{T,h}\in L^2_{\mf_T}(\O;\dbH_0^1),\,f\in L^2_\dbF(\O;L^2(0,T;\dbH_0^1))$, $I_\t$ is a time
 partition of $[0,T]$.
Let  $(Y_h, Z_h)$ be the solution to \eqref{bsde}. Then
\begin{enumerate}[{\rm (i)}]
\item 
\beq
\sum_{n=0}^{N-1}\me\lt[\int_{t_n}^{t_{n+1}}\|Y_h(t)-Y_h(t_n)\|_{\dbL^2}^2\rd t\rt]
\leq C\t \me \lt[\| Y_{T,h}\|_{\dbH_0^1}^2+\int_0^T\|f(t)\|_{\dbL^2}^2\rd t\rt]\, .
\eeq

\item 
%when $\me\|\D_hY_{T,h}\|_{\dbL^2}^2$ is bounded,
Assume further $\sup_{h>0}\me\bigl[\|\D_hY_{T,h}\|_{\dbL^2}^2\bigr]<\infty$. Then
\beq
\sum_{n=0}^{N-1}\me\lt[\int_{t_n}^{t_{n+1}}\|\nb\lt(Y_h(t)-Y_h(t_n)\rt)\|_{\dbL^2}^2\rd t\rt]
\leq C\t \me \lt[\|\D_h Y_{T,h}\|_{\dbL^2}^2+\|\nb Y_{T,h}\|_{\dbH_0^1}^2+\int_0^T\|f(t)\|_{\dbH_0^1}^2\rd t\rt]\, .
\eeq

\item 
%when $\me\|\D_hY_{T,h}\|_{\dbH^1}^2$ is bounded, and $f\in L^2_\dbF(\O;L^2(0,T;\dbH_0^1\cap\dbH^2))$,
Assume further $\sup_{h>0}\me\bigl[\|\D_hY_{T,h}\|_{\dbH_0^1}^2\bigr]<\infty$ and $f\in L^2_\dbF(\O;L^2(0,T;\dbH_0^1\cap\dbH^2))$. Then
\beq
\sum_{n=0}^{N-1}\me\lt[\int_{t_n}^{t_{n+1}}\|\D_h\lt(Y_h(t)-Y_h(t_n)\rt)\|_{\dbL^2}^2\rd t\rt]
\leq C\t \me \lt[\|\D_h Y_{T,h}\|_{\dbH_0^1}^2+\int_0^T\|f(t)\|_{\dbH_2}^2\rd t\rt].
\eeq
\end{enumerate}
Here, the constant $C>0$ only depends on $Y_{T,h},\,f$ and $T$.
\end{lemma}

\begin{proof}  We only prove (i). The other statements can be proved in a similar vein.

By BSDE \eqref{bsde}, we get
\beq
&\sum_{n=0}^{N-1}\me\lt[\int_{t_n}^{t_{n+1}}\|Y_h(t)-Y_h(t_n)\|_{\dbL^2}^2\rd t\rt]
\leq C\t \int_0^T \me \lt[ \|\D_hY_h(t)\|_{\dbL^2}^2+\|\Pi_h^1f(t)\|_{\dbL^2}^2+\|Z_h(t)\|_{\dbL^2}^2 \rt]\rd t.
\eeq
Applying It\^o's formula for $\|Y_h\|_{\dbL^2}^2$ and $\|\nb Y_h\|_{\dbL^2}^2$ in \eqref{bsde}, we see
\beq
\me\lt[\int_0^T\|Z_h(t)\|_{\dbL^2}^2 \rd t \rt]\leq& C\me\lt[ \|Y_{T,h}\|_{\dbL^2}^2+\int_0^T\|\Pi_h^1f(t)\|_{\dbL^2}^2 \rd t \rt],\\
\me\lt[\int_0^T \|\D_hY_h(t)\|_{\dbL^2}^2 \rd t \rt] \leq& C\me\lt[ \|\nb Y_{T,h}\|_{\dbL^2}^2+\int_0^T\|\Pi_h^1f(t)\|_{\dbL^2}^2 \rd t \rt].
\eeq
Then (i) can be deduced by the above estimates.
\end{proof}

\bl{w0913l1}
Suppose that $Y_{T,h}\in \dbD^{1,2}(\dbH_0^1)$, and $f\in L^2_\dbF(\O;L^2(0,T;\dbL^2))$ satisfy
\beq
\sup_{0\leq t\leq T}\me\bigl[\|D_t Y_{T,h}\|_{\dbH_0^1}^2\bigr]+\sup_{0\leq \th\leq T}\sup_{0\leq t\leq T}\me\bigl[\|D_\th D_t Y_{T,h}\|_{\dbL^2}^2\bigr]<C,
\eeq
$$\sup_{0\leq t\leq T}\me \Bigl[\int_t^T\|D_tf(\t)\|^2_{{\mathbb L}^2} \, {\rm d}\t\Bigr] +
\sup_{0\leq \th\leq T}\sup_{0\leq t\leq T}\me \Bigl[\int_{\th\vee t}^T\|D_\th D_tf(\t)\|^2_{{\mathbb L}^2} \, {\rm d}\t\Bigr] \leq C\, ,$$
and for any $s,t\in [0,T]$ with $s\leq t$, 
\begin{equation}\label{ass1}
\me\bigl[\|\lt(D_t-D_s\rt) Y_{T,h} \|^2_{{\mathbb L}^2}\bigr]+\me \Bigl[\int_t^T\| \lt(D_t-D_s  \rt)f(\t) \|^2_{{\mathbb L}^2} \, {\rm d}\t
\Bigr]\leq C|t-s|\, .
\end{equation}
%{\color{green} (Specification of bounds/regularities needed for $Y_T$ and $f$.)}
%{\color{red} $Y_{T,h}\in D^{1,2}(\dbL^2)$ and for any $\th_1,\th_2\in [0,T]$, 
%\beq
%\me\|D_{\th_1}Y_{T,h}-D_{\th_2}Y_{T,h}\|_{\dbL^2}^2\leq |\th_1-\th_2|.
%\eeq }
Then, it holds that
\beq 
{\mathbb E}\bigl[\|Z_h(t)-Z_h(s)\|^2_{{\mathbb L}^2} \bigr]\leq C |t-s|\, .
\eeq
\el

\begin{proof}
By Lemma \ref{malliavin},
we know that  $Z_h(t)=D_tY_h(t)$ for all $0 \leq t \leq T$, and therefore, for $0\leq s\leq t\leq T$,
\bel{w0913e10}
\frac{1}{2} {\mathbb E} \bigl[ \|Z_h(t)-Z_h(s)\|^2_{{\mathbb L}^2}\bigr]\leq \me\bigl[\|D_t Y_h(t)-D_sY_h(t)\|^2_{{\mathbb L}^2}\bigr]+\me \bigl[\|D_s Y_h(t)-D_sY_h(s)\|^2_{{\mathbb L}^2}\bigr]\, .
\ee
In what follows, we estimate the two terms on the right-hand side of \eqref{w0913e10} independently.

\ms

\no {\bf Step 1.}
Fix two $0\leq \th_2\leq \th_1\leq t\leq T$  and define $\delta_\th=D_{\th_1}-D_{\th_2}$.
 By \eqref{w0911e20}, we have the BSDE
\begin{equation}\label{malli-1}
\bal
\delta_\th Y_h(t)-\delta_\th Y_h(T)+ \int_t^T \lt[ -\Delta_h \delta_\th Y_h(s)+\Pi_h^1\d_\th f(s)\rt] \, {\rm d}s
= - \int_t^T \delta_\th Z_h(s)
\, {\rm d}W(s) \,\,\, \forall\, t \in [ \th_1,T]\, .
\eal
\end{equation}
It\^o's formula and Poincar\'{e}'s inequality lead to 
\begin{equation*}
\bal
&\me \bigl[\|\delta_\th Y_h(t)\|^2_{{\mathbb L}^2}\bigr]+
\int_t^T \me\lt[  \Vert \nabla \delta_\th Y_h(s)\Vert^2_{{\mathbb L}^2}  +  \|\delta_\th Z_h(s)\|^2_{{\mathbb L}^2} \rt]\, {\rm d}s\\
&\qquad \leq \me \Bigl[\|\delta_\th Y_h(T)\|^2_{{\mathbb L}^2} + \int_t^T\| \Pi_h^1\d_\th f(s) \|^2_{{\mathbb L}^2} \, {\rm d}s\Bigr]  \, .
\eal
\end{equation*}
%
%Now, we tend to estimate $\me\|\D_\th X_h(T)\|^2$. Since $\D_\th X_h(\cd)$ solves SDE:
%\beq
%\bal
%\D_\th X_h(\t)=&X_h(t)-X_h(s)-\int_s^t A_hD_sX_h(\mu)d\mu+\int_s^t D_sX_h(\mu)d\mu\\
%&-\int_t^{\t}A_h\D_\th X_h(\mu)d\mu+\int_t^\t \D_\th X_h(\mu)dW(\mu)\\
%\equiv &\xi -\int_t^{\t}A_h\D_\th X_h(\mu)d\mu+\int_t^\t \D_\th X_h(\mu)dW(\mu),\, 0\leq t\leq \t\leq T,
%\eal
%\eeq
%by It\^o's formula, we get
%\bel{w0913e11}
%\me\|\D_\th X_h(\t)\|^2\leq C\me\| \xi\|^2,\,0\leq t\leq \t\leq T.
%\ee
%
%Now, we estimate $\me\|\xi\|^2$. 
%
%Firstly, 
%\bel{w0913e12}
%\bal
%&\me\|X_h(t)-X_h(s)\|^2\\
%\leq&2 \me\int_0^T \lt(\|A_hX_h(\t)\|^2+\|b_h(\t)\|^2\rt)d\t|t-s|+2 \me\int_s^t \lt(\|X_h(s)\|^2+\|\si_h(\t)\|^2\rt)d\t\\
%\leq &C \lt(\|\sqrt{A_h}X_h(0)\|^2+\int_0^T\lt( \|b(\t)\|^2+\|\si(\t)\|_{H_0^1}^2 \rt)d\t  \rt)|t-s| \\
%&   +C \lt(\|X_h(0)\|^2+\int_0^T\lt( \|b(\t)\|^2+\|\si(\t)\|^2 \rt)d\t  \rt)|t-s|+\int_s^t \|\si_h(\t)\|^2d\t \\
%\leq & C\lt(\|X(0)\|_{H_0^1}^2+\sup_{\t\in[0,T]}\|\si(\t)\|^2+\int_0^T\lt( \|b(\t)\|^2+\|\si(\t)\|_{H_0^1}^2 \rt)d\t  \rt)|t-s|.
%\eal
%\ee
%
%Second, 
%\bel{w0913e13}
%\bal
%&\me\lt\|\int_s^t A_hD_sX_h(\t)d\t \rt\|^2\\
%\leq& \me\int_s^T \|A_hD_sX_h(\t)\|^2d\t |t-s|\\
%\leq &C\me\| \sqrt{A_h}X_h(s)\|^2|t-s|\\
%\leq &C\lt(\|X(0)\|_{H_0^1}^2+\int_0^T\lt( \|b(\t)\|^2+\|\si(\t)\|_{H_0^1}^2 \rt)d\t  \rt)|t-s|.
%\eal
%\ee
%
%At last,
%\bel{w0913e14}
%\bal
%&\me\lt\|\int_s^t D_sX_h(\t)dW(\t) \rt\|^2\leq \me\int_s^t \|D_sX_h(\t)\|^2d\t \\
%\leq & \me\| X_h(s)\|^2|t-s|\\
%\leq &C \lt(\|X_h(0)\|^2+\int_0^T\lt( \|b(\t)\|^2+\|\si(\t)\|^2 \rt)d\t  \rt)|t-s|
%\eal
%\ee
%
 Taking $\theta_2 = s$ and $\theta_1 = t$ and using \eqref{ass1} then lead to
\begin{eqnarray}\nonumber
\me \bigl[\|D_tY_h(t)-D_sY_h(t)\|^2_{{\mathbb L}^2}\bigr]
&\leq&  \me \Bigl[\|\lt(D_t-D_s \rt) Y_{h,T}\|^2_{{\mathbb L}^2}  + \int_t^T\| \lt(D_t-D_s \rt)  f( \tau) \|^2_{{\mathbb L}^2} \rd \tau\Bigr] \\
\label{w0913e15}
&\leq& C \vert  t-s\vert\, .
\end{eqnarray}
\ms

\no {\bf Step 2.} By \eqref{w0911e20}, It\^o's isometry together with Poincar\'{e}'s inequality,
\bel{w0913e16}
\bal
&\me\bigl[\|D_sY_h(t)-D_sY_h(s)\|^2_{{\mathbb L}^2}\bigr] \\
=&
 \me \Bigl[\bigl\Vert  \int_s^t \lt[-\Delta_h D_sY_h(\t)+\Pi_h^1 D_sf(\t)\rt] \, {\rm d}\t+\int_s^t D_sZ_h(\t)\, {\rm d}W(\t) \bigr\Vert^2_{{\mathbb L}^2}\Bigr]\\ 
\leq&  2|t-s| \, \int_s^T \me  \bigl[ \| \Delta_hD_sY_h(\t)\|^2_{{\mathbb L}^2}+ \|\Pi_h^1 D_sf(\t)\|^2_{{\mathbb L}^2} \bigr]\, {\rm d}\t 
+2\, \int_s^t \me \bigl[\|D_sZ_h(\t)\|^2_{{\mathbb L}^2}\bigr]\, {\rm d}\t \\ 
\leq& C|t-s| \, \me \Bigl[ \|\nabla D_s Y_h(T) \|^2_{{\mathbb L}^2} + \int_s^T\|\Pi_h^1 D_sf(\t)\|^2_{{\mathbb L}^2} \, {\rm d}\t \Bigr]\\
&+C|t-s| \, \sup_{0\leq \th\leq T}\sup_{0\leq t\leq T}\me \Bigl[\|D_\th D_t Y_{T,h}\|_{\dbL^2}^2+\int_{\th\vee t}^T\|D_\th D_tf(\t)\|^2_{{\mathbb L}^2} \, {\rm d}\t\Bigr]\,. 
%\, ,\\ \nonumber
%&&+C\lt(\|X(0)\|^2+\int_0^T\lt( \|b(\t)\|^2+\|\si(\t)\|^2 \rt)d\t  \rt)|t-s|\\ \nonumber
%&\leq&C\lt(\|X(0)\|_{H_0^1}^2+\int_0^T\lt( \|b(\t)\|^2+\|\si(\t)\|_{H_0^1}^2 \rt)d\t  \rt)|t-s|\, .
\eal
\ee
Inserting \eqref{w0913e16} and \eqref{w0913e15} into \eqref{w0913e10} then settles the proof of the lemma.
\end{proof}

%%%%%%%%%%
\subsection{A time-implicit space-time discretization of the
BSPDE \eqref{bshe}}\label{error_bspde2}
We use an implicit time discretization on the mesh $I_\tau$ to approximate {\bf BSPDE}$_h$ \eqref{bsde}; we refer to it as {\bf BSPDE}$_{h\t}$, and the discretization reads as follows: For every $0 \leq n \leq N-1$, find $\bigl( Y^n_h, Z^n_h\bigr) \in L^2_{{\mathcal F}_{t_n}}\bigl( \Omega; {\mathbb V}_h^{1} \times {\mathbb V}_h^{1} \bigr)$ such that
\bel{w0911e10}
\lt\{
\bal
&[\mathds{1} - \tau \Delta_h]Y_h^n = {\mathbb E}\bigl[ Y^{n+1}_h\bigl\vert
{\mathcal F}_{t_n}\bigr] - {\tau} \Pi_h^1 f(t_n)\, ,\\
&Z_h^n = \frac{1}{\tau} {\mathbb E}\Bigl[ Y_h^{n+1} \D_{n+1}W \Big|\mf_{t_n}\Bigr] \\
&Y_h^N=Y_{T,h}\, .
\eal
\rt.
\ee
%%
%where $A_0=\lt(I+A_h|\pi| \rt)^{-1}$.
We introduce an auxiliary BSDE system on each time interval $[t_n, t_{n+1}]$ for the convergence analysis of \eqref{w0911e10}: Find $(\overline{Y}_{h,n}, \overline{Z}_{h,n}) \in
L^2_{\mathbb F}\bigl(\Omega; C([t_{n}, t_{n+1}]; {\mathbb V}^1_h) \times  L^2(t_n, t_{n+1};  {\mathbb V}_h^1 )\bigr)$ such that
%To obtain the convergent rate, we introduce a family of BSDEs: for any $k=0,1,\cds,N-1$,
\bel{w0911e11}
\lt\{
\bal
& \overline{Y}_{h,n}(t)  - \overline{Y}_{h,n}(t_{n+1})  + \int_t^{t_{n+1}} \lt[  -\Delta_h {Y^n_h} + \Pi_h^1 f({t_n})\rt] \,  {\rm d}s= -\int_t^{t_{n+1}}\, \overline{Z}_{h,n}(s) \, {\rm d}W(s)\\
&\overline{Y}_{h,n}(T) =Y_{T,h}\, .
\eal
\rt.
\ee
We now construct $(\overline{Y}_h, \overline{Z}_h) \in L^2_{\mathbb F}\bigl(\Omega; C([0, T]; {\mathbb V}^1_h) \times  L^2(0, T;  {\mathbb V}_h^1 )\bigr)$ via
$\bigl(\overline{Y}_h\bigl\vert_{[t_{n}, t_{n+1}]}, \overline{Z}_h\bigl\vert_{[t_{n}, t_{n+1}]} \bigr) := (\overline{Y}_{h,n}, \overline{Z}_{h,n})$. 
Note that the integrand in the drift is evaluated with the help of 
the solution part $\{ Y^n_h\}_{n=0}^N$ from \eqref{w0911e11}. 
\bl{w0911l1} Let $\{ (Y^n_h, Z^n_h)\}_{n=0}^{N-1}$ solve \eqref{w0911e10}, and
$(\overline{Y}_h, \overline{Z}_h)$ solve \eqref{w0911e11}. For all $0 \leq n \leq N-1$, 
\beq
\bal
Y^n_h = \overline{Y}_h(t_n)  \,, \qquad  {Z^n_h =  \frac{1}{\tau} {\mathbb E}\Bigl[ \int_{t_n}^{t_{n+1}} \overline{Z}_h(s) \, {\rm d}s \Bigl\vert {\mathcal F}_{t_n}
\Bigl]} \, .
\eal
\eeq
\el
\begin{proof}
The first identity is immediate; the second follows from multiplication of \eqref{w0911e11} with the admissible {$ \int_{t_n}^{t_{n+1}} 1\, {\rm d}W(s)$}, and 
application of conditional expectation ${\mathbb E}[\cdot \vert {\mathcal F}_{t_n}]$.
\end{proof}

We may now prove a  strong error estimate for the first component
of $(Y_h, Z_h)$ that solves \eqref{bsde}.
\bt{w0911t2} 
Suppose that $Y_{T,h}\in L^2_{\mf_T}(\O;\dbH_0^1),\,\D_h Y_{T,h}\in L^2_{\mf_T}(\O;\dbL^2),\,f\in L^2_\dbF(\O;L^2(0,T;\dbH_0^1))$ as 
well as 
\beq
\sum_{n=0}^{N-1}\int_{t_n}^{t_{n+1}}{\mathbb E}\bigl[  \Vert f(t) - f(t_n)\Vert^2_{{\mathbb L}^2}\bigr] \leq C\t \, .
\eeq
Let $(Y_h, Z_h)$ solve \eqref{bsde}, and  $\{ (Y^n_h, Z^n_h)\}_{n=0}^{N-1}$ solve \eqref{w0911e10}. %{\color{blue} and the ${\mathbb H}^1$-stability of the ${\mathbb L}^2$-projection $\Pi_h$}; cf. \cite{CrouzeixRaviart}.
There exists a constant $C \equiv (Y_{T,h}, f,T) >0$ such that 
\bel{w1201e1}
\max_{0 \leq n \leq N} {\mathbb E}\bigl[ \Vert Y_h(t_n)-Y_h^n \Vert_{{\mathbb L}^2}^2 \bigr]
+\t \, \sum_{n=0}^{N-1}\me\bigl[\|\nb(Y_h(t_n)-Y_h^n)\|_{\dbL^2}^2\bigr]\leq C\tau\, .
\ee
\et

\begin{proof}
Consider $(\overline{Y}_h, \overline{Z}_h)$ from \eqref{w0911e11}, and define
$\{ e_n\}_{n=0}^{N-1}$, where each $e_n=Y_h(t_n)-\overline{Y}_h(t_n)$ is a ${\mathbb V}_h^1$-valued random variable. Subtracting \eqref{w0911e11} from \eqref{bsde} yields ${\mathbb P}$-a.s.
\begin{eqnarray}\nonumber
e_{n}-e_{n+1} - \int_{t_n}^{t_{n+1}}\Delta_he_n\, {\rm d}s &=& \int_{t_n}^{t_{n+1}} \Delta_h \bigl[Y_h(s)-Y_h(t_n)\bigr] - \Pi_h^1[f(s) - f(t_n)] \, {\rm d}s \\ \label{error_101}
&& +\int_{t_n}^{t_{n+1}} \bigl[Z_h(s)-\overline{Z}_h(s) \bigr] \, {\rm d}W(s)\, .
\end{eqnarray}
Fixing one realization $\omega \in \Omega$, testing with the admissible
$e_n(\omega) \in {\mathbb V}_h^1$, using binomial formula, and then taking expectation, and Poincare's and Young's inequality lead to
\bel{w1208e1}
\bal
&\frac 1 2 {\mathbb E}\Bigl[ \|e_{n}\|^2_{{\mathbb L}^2}-\|e_{n+1}\|^2_{{\mathbb L}^2}+ \Vert e_n-e_{n+1}\Vert^2_{{\mathbb L}^2}+2\int_{t_n}^{t_{n+1}}\Vert \nabla e_n
\Vert_{{\mathbb L}^2}^2 \, {\rm d}s \Bigr]\\ 
&\quad \leq  \int_{t_n}^{t_{n+1}}  \me \Bigl[ \frac{1}{2} \bigl\Vert \nabla \bigl[Y_h(s)-Y_h(t_n)\bigr] \bigr\Vert_{\dbL^2}^2 + \Vert f(s) - f(t_n)\Vert^2_{{\mathbb L}^2} \Bigr] \, {\rm d}s  \\
&\qquad +\frac{1}{2}  \int_{t_n}^{t_{n+1}} \me\Bigl[ \| \nabla e_n\|^2_{\dbL^2}+\|e_n-e_{n+1}\|_{\dbL^2}^2 \Bigr]
\, {\rm d}s+\frac \t 2\me\bigl[\|e_{n+1}\|_{\dbL^2}\bigr]^2\, .
\eal
\ee
Subsequently, the discrete Gronwall  inequality leads to 
\begin{equation}\label{w0911w13}
\bal
\max_{0 \leq n \leq N}  {\mathbb E}\bigl[\Vert e_n \Vert_{{\mathbb L}^2}^2\bigr] 
\leq 
2e^T\, 
 \sum_{n=0}^{N-1}  \int_{t_n}^{t_{n+1}}  \me \Bigl[ \bigl\Vert \nabla \bigl[Y_h(s)-Y_h(t_n)
\bigr] \bigr\Vert^2_{{\mathbb L}^2} + \Vert f(s) - f(t_n)\Vert^2_{{\mathbb L}^2} \Bigr] \, {\rm d}s \, .
\eal
\end{equation} 
Then,  summing up over all steps of \eqref{w1208e1} yields
\bel{w1221e1}
\bal
&\sum_{n=0}^{N-1}\me\lt[\int_{t_n}^{t_{n+1}} \|\nb e_n\|_{\dbL^2}^2\rd s\rt]\\
&\leq \t \sum_{n=0}^{N-1}\me\bigl[\|e_{n+1}\|_{\dbL^2}\bigr]^2
+2\,  
 \sum_{n=0}^{N-1}  \int_{t_n}^{t_{n+1}} \me \Bigl[ \bigl\Vert \nabla \bigl[Y_h(s)-Y_h(t_n)
\bigr] \bigr\Vert^2_{{\mathbb L}^2} + \Vert f(s) - f(t_n)\Vert^2_{{\mathbb L}^2} \Bigr] \, {\rm d}s \, .
\eal
\ee
Then, \eqref{w0911w13}, \eqref{w1221e1} together with (ii) of Lemma \ref{w0911l2}, and \ref{w0911l1} lead to the desired estimate.
\end{proof}

\ms

By Theorems \ref{w0911t1}, \ref{w0911t2} and Lemma \ref{w0911l2} (i), we thus get the following convergence rate for the approximation $\{Y_h^n\}_{n=0}^{N}$ of the first solution component $Y$ to \eqref{vari-1} via
the space-time discretization scheme \eqref{w0911e10}, 
\begin{equation}\label{firstpartbspde}
\max_{0 \leq n \leq N}\me\bigl[\|Y(t_n)-Y_h^n\|^2_{{\mathbb L}^2}\bigr]
+ \sum_{n=0}^{N-1}\int_{t_n}^{t_{n+1}}\me\big[\| \nb \big(Y(t)-Y_h^n\big) \|_{\dbL^2}^2 \big]\rd t
\leq  C \lt(\tau + h^2 \rt)\, .
\end{equation}

\br{w0912r1}

If the drift term of \eqref{bshe} is $-\D Y(t,x)+f \bigl(t,x,Y(t,x)\bigr)$, with a Lipschitz nonlinearity $f$, we may apply a similar procedure to get the above convergence rate. However, the above strategy is not clear to be successful if $Z$ appears in the drift term.

\er

We now derive estimates for the approximation
$\{Z_h^n\}_{n=0}^{N-1}$ of the second solution component $Z$ to \eqref{vari-1}, which uses the characterization $Z_h(t) = D_tY_h(t)$, and \eqref{w0911e20}.

\bt{w0911t2ab} Let $(Y_h, Z_h)$ solve \eqref{vf2}, where data satisfy the assumptions in 
Lemma  \ref{w0911l2} (ii), Lemma \ref{w0913l1}, as well as
\beq
\sum_{n=0}^{N-1}\int_{t_n}^{t_{n+1}}{\mathbb E}\bigl[  \Vert f(t) - f(t_n)\Vert^2_{{\mathbb L}^2}\bigr] \leq C\t \, .
\eeq
Let $\{ (Y^n_h, Z^n_h)\}_{n=0}^{N-1}$ solve \eqref{w0911e10}. 
%{\color{blue} and the ${\mathbb H}^1$-stability of the ${\mathbb L}^2$-projection $\Pi_h^1$}; cf. \cite{CrouzeixRaviart}.
%
%
%
There exists a constant $C \equiv (Y_T, f,T) >0$ such that 
\beq
\sum_{n=0}^{N-1}\int_{t_n}^{t_{n+1}} \me\bigl[\|Z_h(t)-Z_h^n\|^2_{{\mathbb L}^2}\bigr]\, {\rm d}t \leq C \tau\, .
\eeq
\et
The proof begins with an estimate for $Z_h-\overline{Z}_h$, which exploits time regularity properties of the solution 
 $(Y_h, Z_h)$ in stronger norms; cf.~Lemma \ref{w0911l2}, (ii). 
Moreover, the following technical result is needed; see also \cite{Wang16}.

\begin{lemma}\label{conditional}
For any $\varphi \in L^2_{\dbF}(0,T;\dbK)$ and $0\leq s<t \leq T$, define
$$\varphi_0=\frac{1}{t-s}\me \Bigl[ \int_s^t\varphi(\tau) \, {\rm d} \tau\big|\mf_s\Bigr]\, .$$
For any $\xi\in L^2_{\mf_s}(\O;\dbK)$ there holds
$$\me \Bigl[ \int_s^t\|\varphi(\tau)-\varphi_0\|^2_\dbK \, {\rm d}\tau\Bigr] \leq \me\Bigl[\int_s^t\|\varphi(\tau)-\xi\|^2_\dbK  \, {\rm d}\tau\Bigr]\, .$$
\end{lemma}

\begin{proof}
Let $\{\phi_i\}_{i=1}^\infty$ be an orthonormal basis of $\dbK$, and $\Pi_n$ be the projection from $\dbK$
to $\mbox{span}\{\phi_i:i=1,2,\cds,n\}$.
For any $n\in \dbZ$, one has
\begin{equation*}\label{wyyt1}
\begin{aligned}
  &\me \Bigl[\int_s^t\|\varphi(\tau)-\xi\|^2_\dbK\mathrm \, {\rm d}\tau\Bigr]\geq \me
  \Bigl[\int_s^t\|\Pi_n \bigl(\varphi(\tau)-\xi \bigr)\|^2_\dbK\, \mathrm d\tau\Bigr]\\
 =&\me\Bigl[\int_s^t \|\Pi_n\bigl(\varphi(\tau)-\varphi_0\bigr)\|_\dbK^2 +\|\Pi_n(\varphi_0-\xi)\|_\dbK^2  + 2\bigl( \Pi_n\bigl(\varphi(\tau)-\varphi_0\bigr),\Pi_n(\varphi_0-\xi ) \bigr)_{{\mathbb K}}\, \mathrm {\rm d}\tau \Bigr]\\
 =&\me \Bigl[\int_s^t \|\Pi_n(\varphi(\tau)-\varphi_0)\|_\dbK^2 + \|\Pi_n(\varphi_0-\xi)|_\dbK^2\, \mathrm d\tau \Bigr]\\
  &\q+ 2\me \Biggl[\me \Bigl[ \bigl( \int_s^t\Pi_n\varphi(\tau)\, \mathrm d\tau-\me\bigl[\int_s^t\Pi_n\varphi(\tau)\, \mathrm d\tau\big|\mf_s\bigr],\Pi_n(\varphi_0-\xi)  \bigr)_{{\mathbb K}} \big|\mf_s \Bigr] \Biggr] \\
%
%=&\me\int_s^t \|\Pi_n(\varphi(\tau)-\varphi_0)\|_\dbK^2\mathrm d\tau+\me\int_s^t\|\Pi_n(\varphi_0-\xi)|_\dbK^2\mathrm d\tau\\
% \geq&\me\int_s^t \|\Pi_n(\varphi(\tau)-\varphi_0)\|_\dbK^2\mathrm d\tau.
\end{aligned}
\end{equation*}
Since, $\varphi_0$ and $\xi$ are $\mf_s$-measurable, the last term vanishes, i.e.,
\begin{eqnarray*}
&&\me \lt[\me \Bigl[ \bigl( \int_s^t\Pi_n\varphi(\tau)\, \mathrm d\tau-\me\bigl[\int_s^t\Pi_n\varphi(\tau)\, \mathrm d\tau\big|\mf_s\bigr],\Pi_n(\varphi_0-\xi)  \bigr)_{{\mathbb K}} \big|\mf_s \Bigr] \rt] \\
&&\qquad = \me \lt[ \bigl( \me\Bigl[ \int_s^t\Pi_n\varphi(\tau)\, \mathrm d\tau-\me\bigl[\int_s^t\Pi_n\varphi(\tau)\, \mathrm d\tau\big|\mf_s\bigr] \Bigl\vert {\mathcal F}_s\Bigr],\Pi_n(\varphi_0-\xi)  \bigr)_{{\mathbb K}} \rt] = 0\, .\\
%&\me\lt\{\me\lt( \lt\lan \int_s^t\Pi_n\varphi(\tau)\mathrm d\tau-\me\Big(\int_s^t\Pi_n\varphi(\tau)\mathrm d\tau\Big|\mf_s\Big),\Pi_n(\varphi_0-\xi)  \rt\ran \Big|\mf_s\rt) \rt\}\\
%&&\qquad =\me \lt\lan \me\lt[\lt( \int_s^t\Pi_n\varphi(\tau)\mathrm d\tau-\me\Big(\int_s^t\Pi_n\varphi(\tau)\mathrm d\tau\Big|\mf_s\Big)\rt) \Big|\mf_s\rt],\Pi_n(\varphi_0-\xi)  \rt\ran  = 0\,.
\end{eqnarray*}
Therefore,
\beq
\me \Bigl[\int_s^t\|\varphi(\tau)-\xi\|^2_\dbK\, \mathrm d\tau\Bigr] \geq &\me \Bigl[\int_s^t \|\Pi_n(\varphi(\tau)-\varphi_0)\|_\dbK^2\, \mathrm d\tau\Bigr]+\me \Bigl[\int_s^t\|\Pi_n(\varphi_0-\xi)|_\dbK^2\, \mathrm d\tau\Bigr]\\
\geq&\me \Bigl[\int_s^t \|\Pi_n(\varphi(\tau)-\varphi_0)\|_\dbK^2\, \mathrm d\tau\Bigr]\, .
\eeq
By letting $n \uparrow \infty$, we may therefore conclude
\beq
\me \Bigl[\int_s^t \|\varphi(\tau)-\varphi_0\|_\dbK^2\, \mathrm d\tau\Bigr] =\lim_{n\rightarrow \infty}\me\Bigl[\int_s^t \|\Pi_n \bigl(\varphi(\tau)-\varphi_0 \bigr)\|_\dbK^2\, \mathrm d\tau\Bigr]\leq \me\Bigl[\int_s^t\|\varphi(\tau)-\xi\|^2_\dbK\, \mathrm d\tau\Bigr]\, ,
\eeq
which completes the proof.
\end{proof}

\begin{proof}[\bf {Proof of Theorem~\ref{w0911t2ab}}]
{\bf Step 1.}
Claim: theres exists a constant $C$, which is independent of $h,\,\t$, such that
\bel{w1201e2}
\me\Bigl[\int_0^T\|Z_h(s)-\cl Z_h(s)\|^2_{{\mathbb L}^2}\rd s\Bigr]\leq C\t\, .
\ee
%
%The 
%\subsection{error of $Z_h-\overline{Z}_h$ ({\color{red}with $\si(\cd)=0$?})}\label{suse-1}
%

We recall the definition of $\{ e_n\}_{n=0}^{N-1}$ in the proof of Theorem \ref{w0911t2}, as well as equation \eqref{error_101}, which we recast into the form
\begin{eqnarray}\nonumber
&&(\mathds{1} - \tau \Delta_h) e_{n} + \int_{t_n}^{t_{n+1}} \overline{Z}_h(s) - Z_h(s)\, {\rm d}W(s) \\ \nonumber
&&\qquad = e_{n+1}   + \int_{t_n}^{t_{n+1}}\Big[\Delta_h \lt(Y_h(s)-Y_h(t_n)\rt) - \Pi_h^1\lt(f(s) - f(t_n)\rt)\Big]\, {\rm d}s \, .
\end{eqnarray}
%
%By the definition of $e_k$, we can get
%\beq
%\bal
%&(1+A_h|\pi|)e_k+\int_{t_k}^{t_{k+1}}(Z_h(t)-Z_0(t))dW(t)\\
%=&e_{k+1}-\int_{t_k}^{t_{k+1}}\lt(A_h(Y_h(t)-Y_h(t_k))+(f_h(t)-f_h(t_k))\rt)dt,
%\eal
%\eeq
Taking squares and afterwards expectations on both sides, by binomial formula, It\^o isometry, and Young's inequality,
we arrive at
\begin{eqnarray*}
&&{\mathbb E}\Bigl[ \Vert (\mathds{1} - \tau \Delta_h)e_n\Vert^2_{{\mathbb L}^2}
+ \Vert \int_{t_n}^{t_{n+1}} \overline{Z}_h(s) - Z_h(s)\, 
{\rm d}W(s) \Vert^2_{{\mathbb L}^2}\Bigr] \\
&&\quad = {\mathbb E}\Bigl[ \Vert (\mathds{1} - \tau \Delta_h)e_n\Vert^2_{{\mathbb L}^2}
+  \int_{t_n}^{t_{n+1}} \Vert \overline{Z}_h(s) - Z_h(s)\Vert_{{\mathbb L}^2}^2\, 
{\rm d}s \Bigr] \\
&&\quad \leq (1+ \tau) \, {\mathbb E}\Bigl[  \Vert e_{n+1}\Vert^2_{{\mathbb L}^2} + 
(1+ \frac{1}{4\tau})   \tau
   \int_{t_n}^{t_{n+1}} \lt(\bigl\Vert \Delta_h \bigl[Y_h(s)-Y_h(t_n)\bigr] - \Pi_h^1[f(s) - f(t_n)]\bigr\Vert^2_{{\mathbb L}^2}\rt)\, {\rm d}s
\Bigr]\, .
\end{eqnarray*}
Note that $\Vert (\mathds{1} - \tau \Delta_h)e_n\Vert^2_{{\mathbb L}^2} = 
\Vert e_n\Vert^2_{{\mathbb L}^2} + 2 \tau \Vert \nabla e_n\Vert^2_{{\mathbb L}^2} + \tau^2 \Vert \Delta_h e_n\Vert^2_{{\mathbb L}^2}$.
Summation over $0 \leq n \leq N-1$ then leads to
\begin{eqnarray}\nonumber
&&{\mathbb E}\Bigl[ \Vert e_0 \Vert^2_{{\mathbb L}^2} + 2 \tau \sum_{n=0}^{N-1}
\Vert \nabla e_n\Vert^2_{{\mathbb L}^2} +
\int_0^T \Vert \overline{Z}_h(s) - Z_h(s)\Vert^2_{{\mathbb L}^2}\, {\rm d}s \Bigr] \\
\label{addi-l1}
&&\quad \leq \tau \sum_{n=0}^{N-1} {\mathbb E}\bigl[ \Vert e_{n+1}\Vert^2_{{\mathbb L}^2}\bigr] + 2 \sum_{n=0}^{N-1} \int_{t_n}^{t_{n+1}} {\mathbb E}\Bigl[\bigl\Vert \Delta_h \bigl[Y_h(s)-Y_h(t_n)\bigr] - \Pi_h^1[f(s) - f(t_n)]\bigr\Vert^2_{{\mathbb L}^2}
\Bigr]\, {\rm d}s\, .
\end{eqnarray}

By the discrete version of Gronwall's inequality, and  Lemma \ref{w0911l2}, (iii), 
the right-hand side is bounded by $C \tau$. Hence, \eqref{w1201e2} is proved.
%\bt{w0913t1}
%\bel{w0913e6}
%\sum_{k=0}^{N-1}\me\int_{t_k}^{t_{k+1}}\|Z_h(t)-Z_h^\pi(t_k)\|^2dt\leq C |\pi|.
%\ee
%\et
%%%
%
%\begin{proof}

{\bf Step 2.} %We now consider the error $Z_h(t)-Z_h^n$. 
%{\color{blue}Utilizing Lemma \ref{conditional}, one has}
We use the triangle inequality  twice to deduce
\beq
&\sum_{n=0}^{N-1}\me \Bigl[\int_{t_n}^{t_{n+1}}\|Z_h(t)-Z_h^n\|^2_{{\mathbb L}^2}\, {\rm d}t\Bigr]\\
&\qquad \leq 2\sum_{n=0}^{N-1}\int_{t_n}^{t_{n+1}} \me \Bigl[ \|Z_h(t)-\overline{Z}_h(t)\|^2_{{\mathbb L}^2}+ \bigl\| [\overline{Z}_h(t)- Z_h(t)] + [Z_h(t) - Z_h^n] \bigr\|^2_{{\mathbb L}^2}\Bigr] \, {\rm d}t \\
&\qquad \leq 2\sum_{n=0}^{N-1}\int_{t_n}^{t_{n+1}} \me\Bigl[ 3\|Z_h(t)-\overline{Z}_h(t)\|^2_{{\mathbb L}^2} +2\|{Z}_h(t)-Z_h^n\|^2_{{\mathbb L}^2}\Bigl] \, {\rm d}t\, .
\eeq
By the definition of $Z_h^n$,  on taking $\xi = Z_h(t_n)$ in Lemma \ref{conditional} we may further estimate by
\beq
&\quad \leq \sum_{n=0}^{N-1} \int_{t_n}^{t_{n+1}} \me\Bigl[6\|Z_h(t)-\cl Z_h(t)\|^2_{{\mathbb L}^2}+4\|Z_h(t)-Z_h(t_n)\|_{{\mathbb L}^2}^2 \Bigl] \, {\rm d}t\, .
\qquad \qquad 
\eeq
We use \eqref{w1201e2} to bound the first term, and  Lemma \ref{w0913l1} is utilized to bound the last term.
%\begin{eqnarray}\nonumber
%&&\sum_{n=0}^{N-1}\me \Bigl[\int_{t_n}^{t_{n+1}}\|Z_h(t)-Z_h^n\|^2_{{\mathbb L}^2}\, {\rm d}t\Bigr]\\ \label{w0913e7}
%&&\qquad \leq 2\sum_{n=0}^{N-1}\me \Bigl[\int_{t_n}^{t_{n+1}}\|Z_h(t)-\overline{Z}_h(t)\|^2_{{\mathbb L}^2}+\| \overline{Z}_h(t)-Z_h^n\|^2_{{\mathbb L}^2}\, {\rm d}t\Bigr]\\ \nonumber
%&&\qquad \leq 2\sum_{n=0}^{N-1}\me\Bigl[ \int_{t_k}^{t_{k+1}}\|Z_h(t)-\overline{Z}_h(t)\|^2_{{\mathbb L}^2} +\|\overline{Z}_h(t)-Z_h(t_k)\|^2_{{\mathbb L}^2}\, {\rm d}t\Bigl]\\
% && \qquad \leq \sum_{n=0}^{N-1}\me\Bigl[ \int_{t_n}^{t_{n+1}} {\color{red}6\|Z_h(t)-\cl Z_h(t)\|^2_{{\mathbb L}^2}+4\|Z_h(t)-Z_h(t_k)\|_{{\mathbb L}^2}^2}\, {\rm d}t\Bigl]\, .
%\end{eqnarray}
%Inequality (\ref{addi-l1}), in combination with Lemma \ref{w0913l1} then settles the assertion.
%Combining with \eqref{w0913e7} and \eqref{w0913e8}, we can obtain the desired result \eqref{w0913e6}.
%\bel{w0913e9}
%\bal
%&\sum_{k=0}^{N-1}\me\int_{t_k}^{t_{k+1}}\|r_h(t)-r_h^\pi(t_k)\|^2dt\\
%\leq& C\lt(|\pi|+\frac 1\d  \rt)  \lt( \|A^{3/2}p_h(0)\|^2+\| Ap_h(0)\|^2 \rt)|\pi|
%+C\lt( \|\sqrt{A_h}p_h(0)\|^2+\| p_h(0)\|^2 \rt)|\pi|.
%\eal
%\ee
%%
\end{proof}

%%%%%%%%%

\section{Strong rates of convergence for a space-time discretization of SLQ}\label{rate}
%Let $\si\in L^2_\dbF(\O;L^2(0,T;\dbH_0^1))$, $\widetilde{X} \in L^2(0,T; {\mathbb H}^1_0)$ and $X_0 \in {\mathbb H}^1_0 \cap {\mathbb H}^2$ be given.
In this part, we discretize the original problem {\bf SLQ} within two steps, 
starting with its semi-discretization in space (which is referred to as {\bf SLQ}$_{h}$), which is then followed by a
discretization in space and time (which is referred to as {\bf SLQ}$_{h\t}$). Our goal is to prove strong convergence rates in both
cases. By \cite{Lv-Zhang14}, the {\bf SLQ}
problem is uniquely solvable, and its solution $(X^*,U^*)$ may be 
characterized by the following {\bf FBSPDE} with the unique solution
$(X^*, Y, Z, U^*)$, 
\bel{w1205e3}
\lt\{
\bal
&{\rm d} X^*(t)=\bigl(\D X^*(t)+ U^*(t) \bigr){\rm d}t
+\si(t) {\rm d}W(t)\qquad \forall\, t\in (0,T) \, ,\\
&{\rm d} Y(t)=\bigl(-\D Y(t)+ [X^*(t)-\wt X(t)] \bigr){\rm d}t+Z(t){\rm d}W(t) \qquad \forall\,  t\in (0,T) \, ,\\
& X^*(0)=X_0\, ,\qquad Y(T)=-\a \bigl( X^*(T)-\wt X(T) \bigr)\, ,
\eal
\rt.
\ee
with the condition
\bel{w1205e4}
 U^*-Y=0\, .
\ee
We remark that by \eqref{w1205e3}$_1$, $X^*$ may be written as $X^* = {\mathcal S}(U^*)$, 
where 
\beq
{\mathcal S}: L^2_{{\mathbb F}}\bigl(\Omega; L^2(0,T; {{\mathbb L}^2})\bigr) \rightarrow L^2_{{\mathbb F}}\bigl(\Omega; C([0,T]; {\mathbb H}^1_0) \cap
L^2(0,T; {\mathbb H}^2)\bigr)
\eeq
%$${\mathcal S} \in {\mathscr L}\Bigl(L^2_{{\mathbb F}}\bigl(\Omega; L^2(0,T; {{\mathbb L}^2})\bigr), L^2_{{\mathbb F}}\bigl(\Omega; C([0,T]; {\mathbb H}^1_0) \cap
%L^2(0,T; {\mathbb H}^2)\bigr)\Bigr)$$
is the bounded  'control-to-state' map. Moreover, we introduce the  reduced functional
$$\widehat{\mathcal J}: L^2_{{\mathbb F}}\bigl(\Omega; L^2(0,T; {{\mathbb L}^2})\bigr) \rightarrow {\mathbb R} \qquad \mbox{via} \qquad \widehat{\mathcal J}(U ) = {\mathcal J}\bigl(\cS(U), U\bigr)\,.$$
The first  component  of the solution to equation \eqref{w1205e3}$_2$ may be written as $Y = {\mathcal T}(X^*)$, where 
$\cT$ 
\beq
\cT: L^2_{{\mathbb F}}\bigl(\Omega; C([0,T]; {\mathbb L}^2)\bigr) \rightarrow L^2_{{\mathbb F}}\bigl(\Omega; C([0,T]; {\mathbb H}^1_0) \cap L^2(0,T; {\mathbb H}^1_0 \cap {\mathbb H}^2)\bigr),
\eeq
which is also bounded.
%$${\mathcal T} \in {\color{red} {\mathscr L}}\Bigl( L^2_{{\mathbb F}}\bigl(\Omega; C([0,T]; {\mathbb L}^2)\bigr) \rightarrow L^2_{{\mathbb F}}\bigl(\Omega; C([0,T]; {\mathbb H}^1_0) \cap L^2(0,T; {\mathbb H}^1_0 \cap {\mathbb H}^2) 
%\times L^2_{{\mathbb F}}(0,T; {\mathbb H}^1_0)
%\bigr)\Bigr)\, .$$
For every $U \in L^2_{{\mathbb F}}\bigl(\Omega; L^2(0,T; {{\mathbb L}^2})\bigr) $, the Gateaux derivative
$D\h \cJ(U)$  is also a bounded operator on $L^2_{{\mathbb F}}\bigl(\Omega; L^2(0,T; {{\mathbb L}^2})\bigr)$
%$$D \widehat{\mathcal J}(U) \in {\mathscr L}(L^2_{{\mathbb F}}\bigl(\Omega; L^2(0,T; {{\mathbb L}^2}), L^2_{{\mathbb F}}\bigl(\Omega; L^2(0,T; {{\mathbb L}^2})\bigr)$$ of $\widehat{\mathcal J}$ at $U$ then takes the form\footnote{Dec 3rd: Yanqing, are these formal considerations ok? Needs to be checked.}
and takes the form
\begin{equation}\label{derivative-cont}
D \widehat{\mathcal J}(U) = U -{\mathcal T}\bigl({\mathcal S}(U) \bigr)\, .
\end{equation}

%In this section, the following assumptions would be utilized:\\
%{\bf (A1)}: $X_0\in \dbH_0^1$, $\si\in L^2_\dbF(\O;L^2(0;T;{\color{red}\dbH_0^1})),\, \wt X\in L^2(0,T;\dbH_0^1), \, \wt X(T)\in \dbH_0^1\cap\dbH^2$.\\
%{\bf (A2)}: $X_0\in \dbH_0^1\cap \dbH^2,\, \si\in L^2_\dbF(\O;L^2(0,T;\dbH_0^1\cap\dbH^2)), \wt X \in L^2(0,T;\dbH_0^1)$, $\wt X(T)\in \dbH_0^1\cap\dbH^2$, and
%\bel{w1212e2}
%\sum_{n=0}^{N-1}\int_{t_n}^{t_{n+1}}\me\big[\lt\|\si(t)-\si(t_n)\rt\|_{\dbH_0^1}^2\big]+\lt\|\wt X(t)-\wt X(t_n)\rt\|_{\dbL^2}^2
%+\lt\|\wt X(t)-\wt X(t_{n+1})\rt\|_{\dbL^2}^2 \rd t\leq C\t.
%\ee

\subsection{Problem {\bf SLQ}$_{h}$: Semi-discretization in space}\label{rate-1} 

We begin with a spatial semi-discretization  {\bf SLQ}$_h$ of the problem {\bf SLQ} stated in the introduction, 
which reads:  Find an optimal pair $(X_h^*, U^*_h) \in L^2_{{\mathbb F}}\bigl(\Omega; C([0,T]; {\mathbb V}_h^1) \times L^2(0,T; {\mathbb V}_h^0) \bigr)$ that minimizes the functional ($\alpha \geq 0$)
\begin{equation} \label{w1003e2h}
{{\mathcal J}}(X_h, U_h)=\frac 1 2 {\mathbb E} \Bigl[\int_0^T\lt( \Vert X_h(t) -  \widetilde{X}(t) \Vert_{{\mathbb L}^2}^2+ \Vert U_h(t) \Vert^2_{{\mathbb L}^2}\rt)\, {\rm d}t
+ \alpha \Vert X_h(T) -  \widetilde{X}(T) \Vert^2_{{\mathbb L}^2}\Bigr] 
\end{equation}
subject to the equation
\bel{w1013e1a}
\lt\{
\begin{array}{ll}
{\rm d}X_h(t)=\bigl[\Delta_h X_h(t)+\Pi_h^1 U_h(t)\bigr]\, {\rm d}t+ \Pi_h^1 \si(t) {\rm d}W(t) \q &\forall\, t \in [0,T]\,,\\
%X(t,x)=0,\q\q & (t,x)\in (0,T)\times \partial D,\\
X_h(0)= \Pi_h^1 X_0\,.
\end{array}
\rt.
\ee
The existence of a unique optimal pair $(X_h^*, U^*_h)$ follows from \cite{Yong-Zhou99}, as well as its characterization via Pontryagin's maximum principle, i.e.,
\begin{equation}\label{pontr1a} 
0 = U^*_h(t) - \Pi^0_hY_h(t) \qquad \forall\, t \in (0,T)\, ,
\end{equation}
where the adjoint $(Y_h, Z_h) \in L^2_{{\mathbb F}}\bigl( \Omega; C([0,T]; {\mathbb V}_h^1)\bigr) \times L^2_{{\mathbb F}}\bigl( \Omega; L^2(0,T; {\mathbb V}_h^1)\bigr)$ solves the {\bf BSPDE}$_h$
\bel{bshe1a}
\lt\{
\begin{array}{ll}
{\rm d}Y_h(t)= \Bigl[-\Delta_h Y_h(t)+ \bigl[X^*_h(t) - \Pi^1_h\widetilde{X}(t)\bigr] \Bigr]{\rm d}t+Z_h(t) {\rm d}W(t)  \q & \forall\, t \in [0,T]\, ,\\
%Y (t)=0    & \mbox{on }[0,T]\, ,\\
Y_h(T)= -\alpha \bigl(X^*_h(T)- \Pi^1_h\wt X(T)\bigr)\,.
\end{array}
\rt.
\ee
In \cite{Dunst-Prohl16}, optimal error estimates have been obtained for 
$(X_h^*,Y_h, Z_h)$ with the help of a fixed point argument --- which crucially exploits $T>0$ to be {\em sufficiently small}. The goal in this section is to derive
corresponding estimates for $(X_h^*,U_h^*, Y_h, Z_h)$ for {\em arbitrary} $T>0$ via a variational argument which 
exploits properties of the reduced functional $\widehat{\cJ}
\equiv \widehat{\cJ}(u)$ that is now defined: once an estimate for $\int_0^T {\mathbb E}\bigl[\Vert (U^* - U^*_h)(s)\Vert^2_{{\mathbb L}^2}\bigr]\, {\rm d}s$ has been obtained, 
we use the convergence analysis from Section \ref{error_bspde} to derive estimates for 
$X^* - X^*_h$, as well as $Y - Y_h$ and $Z -Z_h$.

By the unique solvability property of \eqref{w1013e1a}, we associate to this equation the  bounded solution operator
$${\mathcal S}_h: L^2_{{\mathbb F}}\bigl(\Omega; L^2(0,T; \dbV_h^0)\bigr) \rightarrow L^2_{{\mathbb F}}\bigl(\Omega; C([0,T]; {\mathbb V}^1_h)\bigr)\,,$$
which allows to introduce the reduced functional
\begin{equation}\label{red-conv}\widehat{\cJ}_h: L^2_{{\mathbb F}}\bigl(\Omega; L^2(0,T; \dbV_h^0)\bigr) \rightarrow {\mathbb R}\,, \qquad \mbox{via} \qquad
\widehat{\cJ}_h (U_h) = \cJ \bigl({\mathcal S}_h(U_h), U_h\bigr)\, .
\end{equation}
The first solution component to equation \eqref{bshe1a} may be written as $Y_h = {\mathcal T}_h(X^*_h)$, where 
\beq
\cT_h: L^2_{{\mathbb F}}\bigl(\Omega; C([0,T]; {\mathbb V}^1_h) \bigr) \rightarrow L^2_{{\mathbb F}}\bigl(\Omega; C([0,T]; {\mathbb V}^1_h) \bigr)\, .
\eeq
%$${\mathcal T}_h \in {\color{red} {\mathscr L}}\Bigl( L^2_{{\mathbb F}}\bigl(\Omega; C([0,T]; {\mathbb L}^2)\bigr), L^2_{{\mathbb F}}\bigl(\Omega; C([0,T]; {\mathbb V}^1_h) 
%%\times L^2_{{\mathbb F}}(0,T; {\mathbb H}^1_0)
%\bigr)\Bigr)\, .$$
For every $U_h \in L^2_{{\mathbb F}}\bigl(\Omega; L^2(0,T; {{\mathbb V}^0_h})\bigr) $, the Gateaux derivative 
${D\h\cJ_h(U_h)}$ is a bounded operator (uniformly in $h$) on $L^2_{{\mathbb F}}\bigl(\Omega; L^2(0,T; {{\mathbb V}^0_h})\bigr) $  
%$$D \widehat{\mathcal J}_h(U_h) \in {\mathscr L}(L^2_{{\mathbb F}}\bigl(\Omega; L^2(0,T; {{\mathbb V}^0_h}), L^2_{{\mathbb F}}\bigl(\Omega; L^2(0,T;{{\mathbb V}^0_h})\bigr)$$ 
%of $\widehat{\mathcal J}_h$ 
at $U_h$, and has the form
\begin{equation}\label{derivative-semidisc}
D {\widehat\cJ}_h (U_h)=U_h-\Pi_h^0 {\mathcal T}_h\bigl({\mathcal S}_h (U_h)\bigr)\, .
\end{equation}
Let $U_h \in L^2_{{\mathbb F}}\bigl(\Omega; L^2(0,T; {{\mathbb V}^0_h})\bigr)$ be arbitrary; it is due to the quadratic structure of the reduced functional \eqref{red-conv} that
$$ {\mathbb E} \Bigl[\bigl( D^2 \widehat{\mathcal J}_h(U_h) R_h,R_h \bigr)_{L^2(0,T;\dbL^2)} \Bigr]\geq {\mathbb E}\bigl[\Vert R_h\Vert^2_{L^2(0,T; {\mathbb L}^2)}\bigr] \qquad \forall\, R_h \in 
L^2_\dbF\bigl( \Omega; L^2(0,T; {\mathbb V}^0_h)\bigr)\, .$$
As a consequence, on putting $R_h = U^*_h - \Pi^0_h U^*$,
\bel{esti-control}
\bal
%\begin{eqnarray}\nonumber
&{\mathbb E}\bigl[ \Vert U^*_h - \Pi^0_h U^*\Vert_{L^2(0,T; {\mathbb L}^2)}^2\bigr]
\leq {\mathbb E}\Bigl[ \Bigl( D^2 \widehat{\mathcal J}_h(U_h) (U^*_h - \Pi^0_h U^*),U^*_h - \Pi^0_h U^*\bigr)_{L^2(0,T;\dbL^2)}\Bigr]\\ 
&\qquad = {\mathbb E}\Bigl[ \bigl( D\widehat{\mathcal J}_h(U^*_h) ,U^*_h - \Pi^0_h U^*\bigr)_{L^2(0,T;\dbL^2)} - 
\bigl( D\widehat{\mathcal J}_h(\Pi^0_h U^*) ,U^*_h - \Pi^0_h U^*\bigr)_{L^2(0,T;\dbL^2)}\Bigr]\, .
%\end{eqnarray}
\eal
\ee
Note that $D\widehat{\mathcal J}_h(U^*_h) = 0$ by \eqref{pontr1a}, as well as $D \widehat{\mathcal J}(U^*) = 0$ by \eqref{w1205e4}, such that the last line equals
\bel{bound1}
\bal
%\begin{eqnarray}\nonumber 
=& {\mathbb E}\Bigl[ \bigl( D\widehat{\mathcal J}(U^*) ,U^*_h - \Pi^0_h U^*
\bigr)_{L^2(0,T;\dbL^2)} - \bigl( D \widehat{\mathcal J}
(\Pi_h^0 U^*),U^*_h - \Pi^0_h U^* \bigr)_{L^2(0,T;\dbL^2)}\Bigr] \\ 
&+ {\mathbb E} \Bigl[\bigl( D \widehat{\mathcal J}
(\Pi^0_h U^*), U^*_h - \Pi^0_h U^*\bigr)_{L^2(0,T;\dbL^2)} 
- \bigl( D\widehat{\mathcal J}_h(\Pi^0_h U^*) ,U^*_h - \Pi^0_h U^*\bigr)_{L^2(0,T;\dbL^2)}\Bigr]\\
 =:& I + II\, .
%\end{eqnarray}
\eal
\ee
We use \eqref{derivative-cont} to bound $I$ as follows,
\beq
I =& {\mathbb E}\Bigl[\Bigl( U^* - \Pi^0_h U^* + {\mathcal T}\bigl({\mathcal S}(\Pi^0_h U^*) \bigr) -{\mathcal T}\bigl({\mathcal S}(U^*) \bigr), U^*_h - \Pi_h^0 U^*\Bigr)_{L^2(0,T;\dbL^2)}\Bigr] \\
\leq& \lt(\Bigl( {\mathbb E}\bigl[\Vert U^* - \Pi^0_h U^* \Vert_{L^2(0,T; {\mathbb L}^2)}^2\bigr] \Bigr)^{1/2} + I_a \rt)   \Bigl( {\mathbb E}\bigl[ \Vert U^*_h - \Pi_h^0 U^*\Vert_{L^2(0,T; {\mathbb L}^2)}^2\bigr]\Bigr)^{1/2}\,,
\eeq
where $I^2_a = {\mathbb E}\bigl[ \Vert {\mathcal T}\bigl( {\mathcal S}(U^*) - {\mathcal S}(\Pi^0_h U^*) \bigr)\Vert_{L^2(0,T;{\mathbb L}^2)}^2\bigr]$. By Poincar\'{e}'s inequality,  and a stability bound (see also \eqref{vari-2}) for the backward stochastic heat equation \eqref{bshe}, as well as for the stochastic heat equation \eqref{forw1} (see also \eqref{stochheat1}),
\begin{eqnarray}\nonumber I^2_a &\leq& C  {\mathbb E}\bigl[\Vert \bigl({\mathcal S}(U^*) - {\mathcal S}(\Pi^0_h U^*)\bigr)(T)\Vert^2_{{\mathbb L}^2} +  \Vert {\mathcal S}(U^*) - {\mathcal S}(\Pi^0_h U^*)\Vert^2_{L^2(0,T; {\mathbb L}^2)}\bigr]\\ \label{estima}
&\leq& C {\mathbb E}\bigl[\Vert U^* - \Pi^0_h U^*\Vert^2_{L^2(0,T; {\mathbb L}^2)}\bigr]\, .
\end{eqnarray} 
By optimality condition \eqref{w1205e4}, and the regularity properties of the solution
to {\bf BSPDE} \eqref{bshe}, we know that already
$U^* \in L^2_{\mathbb F}\bigl(\Omega; C([0,T]; {\mathbb H}^1_0) \cap L^2(0,T; {\mathbb H}^1_0 \cap {\mathbb H}^2)\bigr)$; as a consequence, the right-hand side of \eqref{estima} may be bounded by $Ch^2$.

We use the representation \eqref{derivative-semidisc} and properties of the projection
$\Pi_h^0$ to bound $II$ via
\begin{eqnarray*}
II &=& {\mathbb E}\Bigl[ \Bigl( {\mathcal T}\bigl( {\mathcal S}(\Pi_h^0 U^*)\bigr) - \Pi_h^0 {\mathcal T}_h\bigl( {\mathcal S}_h(\Pi_h^0 U^*)\bigr), U^*_h - \Pi^0_h U^*\Bigr)_{L^2(0,T; {\mathbb L}^2)}\Bigr]\\
&\leq& II_a \times
\Bigl({\mathbb E}\bigl[ \Vert U^*_h - \Pi^0_h U^* \Vert^2_{L^2(0,T; {\mathbb L}^2)}\bigr]\Bigr)^{1/2}\,,
\end{eqnarray*}
where $II^2_a:= {\mathbb E}\bigl[ \Vert {\mathcal T}\bigl( {\mathcal S}(\Pi_h^0 U^*)\bigr) - {\mathcal T}_h\bigl( {\mathcal S}_h(\Pi_h^0 U^*)\bigr)\Vert^2_{L^2(0,T; {\mathbb L}^2)}\bigr]$. We split $II^2_a$ into two terms
\begin{eqnarray*} II^2_{a,1} &=& {\mathbb E}\bigl[ \Vert {\mathcal T}\bigl( {\mathcal S}(\Pi_h^0 U^*)\bigr) - {\mathcal T} \bigl( {\mathcal S}_h(\Pi_h^0 U^*)\bigr)\Vert^2_{L^2(0,T; {\mathbb L}^2)}\bigr] \\
\mbox{and} \qquad II^2_{a,2} &=&  {\mathbb E}\bigl[ \Vert {\mathcal T} \bigl( {\mathcal S}_h(\Pi_h^0 U^*)\bigr) - {\mathcal T}_h\bigl( {\mathcal S}_h(\Pi_h^0 U^*)\bigr)\Vert^2_{L^2(0,T; {\mathbb L}^2)}\bigr]\, .
\end{eqnarray*}
In order to bound $II^2_{a,1}$, we use stability properties for {\bf BSPDE} \eqref{bshe}, in combination with the error estimate \eqref{esti-space1} for \eqref{sde} to conclude
$$II^2_{a,1} \leq C {\mathbb E}\bigl[ \bigl\Vert {\mathcal S}(\Pi_h^0 U^*) - {\mathcal S}_h(\Pi_h^0 U^*) \bigr\Vert^2_{L^2(0,T; {\mathbb L}^2)}\bigr] \leq Ch^2\, . $$
In order to bound $II^2_{a,2}$, we use the error estimate in Theorem \ref{w0911t1} for {\bf BSPDE} \eqref{bshe}, in combination
with stability properties of \eqref{sde}, and again the error estimate \eqref{esti-space1} for \eqref{sde} to find
$$II^2_{a,2} \leq C \Bigl({\mathbb E}\bigl[ \bigl \Vert{\mathcal S}(\Pi_h^0 U^*)(T) - {\mathcal S}_h(\Pi_h^0 U^*)(T) \bigr \Vert^2_{{\mathbb L}^2} \bigr] +h^2\Bigr) \leq Ch^2\, .$$
We now insert these estimates into \eqref{bound1} resp.~\eqref{esti-control} to obtain the bound
$${\mathbb E}\bigl[ \Vert U^*_h - \Pi^0_h U^*\Vert_{L^2(0,T; {\mathbb L}^2)}^2\bigr] \leq Ch^2\, . $$ 
By arguing as below \eqref{estima}, this settles part (i) of the following 
%
%\begin{lemma}[\cite{Malanowski82}, Lemma 1.1]\label{w1201l2}
% There exists a linear invertible operator $\sK: L^2_\dbF\bigl(\Omega; L^2(0,T;\dbV_h^0)\bigr)\rightarrow L^2_\dbF\bigl(\Omega; L^2(0,T;\dbV_h^0)\bigr)$, 
% such that
%\bel{w1003e24}
%D\widehat\cJ_h(U^*_h)-D\widehat\cJ_h(\Pi_h^0U^*)=\sK \lt(U^*_h -\Pi_h^0 U^*\rt)
%\ee
%with
%\bel{w1003e25}
%|\sK^{-1}|\leq \g,
%\ee
%where $\g$ does not depend on $h$.
%%
%\el
%
%We are now in the position to prove the convergence rate of spatial discretization. 

\bt{rate1}
%Suppose that assumption {\rm (A1)} holds.
Let $(X^*,Y, Z, U^*)$ be the solution to problem {\bf SLQ}, and $(X^*_h,Y_h, Z_h, U^*_h)$ be the solution to problem {\bf SLQ$_h$}. There exists $C \equiv C(X_0, T)>0 $ such that
%\begin{eqnarray*}
%%%{w1003e39}
%{\rm (i)} && \me \Bigl[\int_0^T\|U^*(t)-U^*_h(t)\|_{\dbL^2}^2\, {\rm d} t\Bigr] \leq C h^2\, , \\
%{\rm (ii)}&& \sup_{0 \leq t \leq T} \me\lt(\|X^*(t)\|_{\dbL^2}^2+\|Y(t)\|_{\dbL^2}^2\rt)
%+\me\int_0^T\lt(\|X^*(t)\|_{\dbH_0^1}^2+\|Y(t)\|_{\dbH_0^1}^2+\|Z(t)\|_{\dbL^2}^2  \rt)\rd t \leq Ch^2\, .
%\end{eqnarray*}
%%
\begin{eqnarray*}
{\rm (i)} && \me \Bigl[\int_0^T\|U^*(t)-U^*_h(t)\|_{\dbL^2}^2\, {\rm d} t\Bigr] \leq C h^2\, ; \\ 
{\rm (ii)} &&  \sup_{0 \leq t \leq T} {\mathbb E}\Bigl[ \|X^*(t)-X_h^*(t)\|_{\dbL^2}^2   + \int_0^T {\mathbb E}\Bigl[ \|X^*(t)-X_h^*(t)\|_{\dbH_0^1}^2\Bigr]\rd t \leq Ch^2\,;\\
{\rm (iii)} &&
  \sup_{0 \leq t \leq T} \me \bigl[ \|Y(t)-Y_h(t)\|_{\dbL^2}^2 \bigr]+  \int_0^T {\mathbb E}\Bigl[ \|Y(t)-Y_h(t)\|_{\dbH_0^1}^2+\|Z(t)-Z_h(t)\|_{\dbL^2}^2 \Bigr] \rd t \leq Ch^2\, .
\end{eqnarray*}

\et
%%%%

\begin{proof}
Since $U^* \in L^2_{\dbF}(\O;L^2(0,T;\dbH_0^1))$, and (i),
the first estimate of (ii) can be deduced as \eqref{esti-space1}.
Assertion (iii) now follows accordingly as  Theorem \ref{w0911t1}, thanks to (ii).
\end{proof}

\subsection{Problem {\bf SLQ}$_{h\tau}$: Discretization in space and time}\label{rate-2}

In this part, we provide the temporal discretization  of problem {\bf SLQ$_h$} which was analyzed in Section \ref{rate-1}. For this purpose, we use
a mesh $I_\tau$ covering $[0,T]$, and consider processes $(X_{h\tau}, U_{h\tau}) \in {\mathbb X}_{h\t} \times {\mathbb U}_{h\t}\subset L^2_\dbF\bigl(\O;L^2(0,T;\dbV_h^1)\bigr)\times L^2_\dbF\bigl(\O;L^2(0,T;\dbV_h^0)\bigr)$, where
 \begin{eqnarray*}
 {\mathbb X}_{h\t} &:=& \lt\{X \in L^2_\dbF\bigl(\O;L^2(0,T;\dbV_h^1)\bigr): \ X(t)=X(t_n), \,\,   \forall t\in [t_n, t_{n+1}),  \,\, n=0,1,\cds, \, N-1\rt\}\, ,\\
{\mathbb U}_{h\t} &:=& \lt\{U \in L^2_\dbF\bigl(\O;L^2(0,T;\dbV_h^0) \bigr):\ U(t)=U(t_n),  \,\,    \forall t\in [t_n, t_{n+1}),  \,\, n=0,1,\cds, \, N-1\rt\}\, ,
\end{eqnarray*}
and for any $X\in \dbX_{h\t},\,U\in \dbU_{h\t}$,
\beq
\|X\|_{\dbX_{h\t}}:=\bigg(\t \sum_{n=1}^N\me\big[\|X(t_n)\|_{\dbL^2}^2\big]\bigg)^{1/2},\,\,
\|U\|_{\dbU_{h\t}}:=\bigg(\t \sum_{n=0}^{N-1}\me\big[\|U(t_n)\|_{\dbL^2}^2\big]\bigg)^{1/2}.
\eeq
Problem {\bf SLQ$_{h\t}$} then reads as follows: Find an optimal pair 
$(X_{h\t}^*,U_{h\t}^*)\in {\mathbb X}_{h\t} \times {\mathbb U}_{h\t}$ which minimizes the cost functional 
\bel{w1003e8}
\bal
&\cJ_{\t}(X_{h\t}, U_{h\t})\\
&=\frac {\t} 2\sum_{n= 1}^N \me\bigl[ \Vert X_{h\t}(t_n)- \wt X(t_n) \Vert_{\dbL^2}^2\bigr]
+\frac {\t} 2\sum_{n=0}^{N-1} \me \bigl[\Vert U_{h\t}(t_n) \Vert_{\dbL^2}^2\bigr]+\frac\a 2 \me \bigl[\| X_{h\t}(T)-\wt X(T) \|_{\dbL^2}^2\bigr]\, ,
\eal
\ee
subject to the difference equation
\bel{w1003e7}
\lt\{
\bal
& X_{h\t}(t_{n+1})-X_{h\t}(t_n)= \tau \bigl[\D_h X_{h\t}(t_{n+1})+
{\Pi_h^1U_{h\t}(t_n)} \bigr]+\Pi_h^1\si(t_n) \D_{n+1}W \\
 &\qq\qq\qq\qq\qq\qq\qq n=0,1,\cds,N-1\, ,\\
& X_{h\t}(0)= {\Pi_h^1 X_0}\,,
\eal
\rt.
\ee
%where 
%\beq
%\cU_{h\t}\equiv\lt\{u(\cd)\in L^2_\dbF(\O;L^2(0,T;\dbV_h^0)) |u(t)=u(t_k),  \forall t\in [t_k, t_{k+1}),  k=0,1,\cds, N-1\rt\},
%\eeq
%%
% \beq
% \cX_{h\t}\equiv \lt\{X(\cd)\in L^2_\dbF(\O;L^2(0,T;\dbV_h^1)) |X(t)=X(t_k),  \forall t\in [t_k, t_{k+1}),  k=0,1,\cds, N-1\rt\}.
% \eeq
%%
where $\Delta_{n+1} W = W(t_{n+1}) - W(t_n)$.
The following result states the Pontryagin maximum principle for problem {\bf SLQ$_{h\tau}$}, which is later used to verify
convergence rates for the solution to problem {\bf SLQ$_{h\t}$} towards
the solution to  {\bf SLQ}.

\bt{MP}
%Under assumption {\rm (A1)}, 
Problem {\bf SLQ$_{h\t}$} admits a unique minimizer $(X^*_{h\t}, U^*_{h\t})\in {\mathbb X}_{h\tau} \times {\mathbb U}_{h\tau}$, which is (part of) the unique solution $$(X^*_{h\tau}, Y_{h\tau}, U^*_{h\tau}) \in \bigl[{\mathbb X}_{h\tau} \bigr]^2  \times {\mathbb U}_{h\tau}$$ to the following 
forward-backward difference equation for
$0 \leq n \leq N-1$,
{\bel{w1212e3}
\lt\{
\bal
& [\mathds{1} - \tau \Delta_h]X^*_{h\t}(t_{n+1})= X^*_{h\t}(t_n)+ \tau \Pi_h^1U^*_{h\t}(t_n) +\Pi_h^1\si(t_n) \D_{n+1}W\, ,\\
&[\mathds{1} - \tau \Delta_h]Y_{h\t}(t_n) = {\mathbb E}\lt[ Y_{h\t}(t_{n+1})- {\tau} \bigl(X^*_{h\t}(t_{n+1})-\Pi_h^1\wt X(t_{n+1})\bigr)\bigl\vert
{\mathcal F}_{t_n}\rt]\, , \\
%&Z_{h\t}(t_n) = \frac{1}{\tau} {\mathbb E}\Bigl[ \lt(Y_{h\t}(t_{n+1}) -\t  \lt(X^*_{h\t}(t_{n+1})-\Pi_h^1\wt X(t_{n+1})\rt)\rt) \D_{n+1}W \Big|\mf_{t_n}\Bigr],\,n=0,1,\cds,N-1, \\
&X^*_{h\t}(0)=\Pi_h^1 X_0\, ,\qquad Y_{h\t}(T)=-\a\bigl(X^*_{h\t}(T)-\Pi_h^1\wt X(T)\bigr)\, ,
\eal
\rt.
\ee}
%%
%{\color{red}\bel{w1003e11}
%\lt\{
%\bal
%& X^*_{h\t}(t_{n+1})- X^*_{h\t}(t_n)= \tau \bigl[\D_hX^*_{h\t}(t_{n+1})+\Pi_h^1U^*_{h\t}(t_n) \bigr]+\Pi_h^1\si(t_n) \D_{n+1}W\, ,\\
%%\,    n=0,1,\cds, N-1,\\
%&Y_{h\t}(t_{n+1})-Y_{h\t}(t_n)= \tau \lt[-\D_hY_{h\t}(t_{n})+\lt(X^*_{h\t}(t_{n+1})-\Pi_h^1\wt X(t_{n+1})\rt) \rt]\\
%&\qq\qq\qq\qq\qq+\int_{t_n}^{t_{n+1}}Z_{h\t}(t)\, {\rm d}W(t) \q n=0,1,\cds,N-1\, ,\\
%&  X^*_{h\t}(0)=\Pi_h^1 X_0\, ,\qquad Y_{h\t}(T)=-\a\lt(X^*_{h\t}(T)-\Pi_h^1\wt X(T)\rt)\, ,
%\eal
%\rt.
%\ee}
together with 
\bel{w1003e12}
U^*_{h\t}(t_n)-\Pi_h^0Y_{h\t}(t_n)=0 \qquad  n=0,1,\cds,N-1 \, .
\ee
%admits a unique solution $\ds\lt(X^*_{h\t}(\cd), Y_{h\t}(\cd), Z_{h\t}(\cd)\rt)\in \cX_{h\t} \times \cX_{h\t} \times L^2_\dbF(\O;L^2(0,T;\dbV_h^1))$, where $X^*_{h\t}(\cd)$ is the optimal state.
%
\et

By \eqref{w1003e12}, we can see that $U_{h\t}^*$ is c\`adl\`ag, and then $U_{h\t}^*\in \dbU_{h\t}$.
Inserting \eqref{w1003e12} into \eqref{w1212e3}$_1$ leads to a coupled problem for $\lt( \lt\{ X^*_{h\tau}(t_{n+1})\rt\}_{n=0}^{N-1}, \lt\{ Y_{h\tau}(t_{n}) \rt\}_{n=0}^{N-1}\rt)\,,$ where \eqref{w1212e3}$_2$ is similar to
\eqref{w0911e10}. Note that no $Z$-component appears explicitly in \eqref{w1212e3}$_2$, where the conditional 
expectation is used to compute the $Y$-component. It is in particular due to the need to compute conditional expectations in \eqref{w1212e3}$_2$ that the optimality system \eqref{w1212e3}--\eqref{w1003e12} is still not amenable to an actual implementation, but serves as a key step towards a practical method which approximately solves {\bf SLQ}$_{h\t}$ --- which is proposed and studied in Section \ref{numopt}.
\begin{proof}
We divide the proof into three steps.
\ms

\no{\bf Step 1.} Let $\ds A_0 :=\lt(\mathds{1}-\tau \D_h\rt)^{-1}$.
For any $U_{h\t} \in {\mathbb U}_{h\t}$, by equation \eqref{w1212e3}$_1$, we 
have
\begin{equation}\label{w1003e13}
X_{h\t}(t_n)=A_0\Bigl(X_{h\t}(t_{n-1})+\tau \Pi_h^1 U_{h\t}(t_{n-1}) +\Pi_h^1\si(t_{n-1})\D_nW\Bigr)\, .
\end{equation}
Hence, by iteration  we arrive at
\begin{eqnarray}\nonumber
X_{h\t}(t_n)&=&A_0^nX_{h\t}({0})+\tau \sum_{j=0}^{n-1}A_0^{n-j}\Pi_h^1U_{h\t}(t_j) +\sum_{j=1}^nA_0^{n+1-j}\Pi_h^1\si(t_{j-1})\D_jW\\ \label{rep1}
&=:& (\G {\Pi_h^1 X_0})(t_{n})+(LU_{h\t} )(t_{n})+f(t_n)\, .
\end{eqnarray}
Here, $\G : \dbV_h^1\rightarrow \dbX_{h\t}$ and $L: \dbU_{h\t}\rightarrow \dbX_{h\t}$ are bounded operators.
%%
%\beq
%\G : \dbV_h^1\rightarrow \dbX_{h\t},\,\,\mbox{and } L: \dbU_{h\t}\rightarrow \dbX_{h\t}.
%\eeq
%%
Below, we use the abbreviations
\begin{equation}\label{w1003e14}
\h \G {\Pi_h^1 X_0} := \G{\Pi_h^1 X_0}(T)\, , 
\qquad \h L{U_{h\tau}} := (LU_{h\tau})(T)\,, \qquad  \h f=f(T)\, .
\end{equation}

\ms

{\bf Claim}: For any $\xi\in {\mathbb X}_{h\t}$, and any $\eta\in L^2_{\mf_T}(\O;\dbV_h^1)$,
\bel{w1003e01}
L^*\xi=-\Pi_h^0Y_0\,,\qquad \h L^* \eta= -\Pi_h^0 Y_1\, ,
\ee
where $(Y_0,Z_0)$ solves the following backward stochastic difference equation:
\bel{w1206e1}
\lt\{
\bal
&Y_0(t_{n+1})-Y_0(t_n)= \tau \bigl(-\D_hY_0(t_n)+\xi(t_{n+1}) \bigr) +\int_{t_n}^{t_{n+1}}Z_0(t)\rd W(t) \, \quad n=0,1,\cds, N-1\, ,\\
&Y_0(t_N)=Y_0(T)=0\, ,
\eal
\rt.
\ee
and  $(Y_1,Z_1)$ solves
\beq
\lt\{
\bal
&Y_1(t_{n+1})-Y_1(t_n)=-\tau \D_hY_1(t_n)+\int_{t_n}^{t_{n+1}}Z_1(t)\rd W(t)\qquad n=0,1,\cds, N-1\, ,\\
&Y_1(T)=-\eta\, .
\eal
\rt.
\eeq

{\em Proof of Claim:} The existence and the uniqueness of solutions to \eqref{w1206e1} are obvious.
Note that
\begin{equation}\label{w1003e15}
Y_0(t_j)=\me \bigl[A_0Y_0(t_{j+1})- \tau A_0\xi(t_{j+1}) \big|\mf_{t_j} \bigr].
\end{equation}
With the similar procedure as that in 
\eqref{w1003e13}, we conclude from \eqref{w1003e15} and \eqref{w1206e1}$_2$,
\bel{w1003e15a}
\bal
Y_0(t_j)
=&\me \bigl[ A_0^{N-j}Y_0(t_N)\big|\mf_{t_j} \bigr]-  \me \bigl[ \tau \sum_{k=j+1}^NA_0^{k-j}\xi(t_k)\big|\mf_{t_j} \bigr]\\
=&- \me\Bigl[ \tau \sum_{k=j+1}^NA_0^{k-j}\xi(t_k)  \big|\mf_{t_j} \Bigr]\, .
\eal
\ee

{Let $U_{h\t} \in {\mathbb U}_{h\t}$ be arbitrary}. By the definition of $L$, \eqref{w1003e15a} and the fact 
$A_0=A_0^\top$, we can calculate that
\begin{eqnarray*}
&&\tau \sum_{n=1}^N \me \Bigl[\bigl( (LU_{h\t})(t_n),\xi(t_n)  \bigr)_{ {\mathbb L}^2}\Bigr] 
=\tau \sum_{n=1}^N \me \Bigl[ \bigl( \tau \sum_{j=0}^{n-1}A_0^{n-j}\Pi_h^1 U_{h\t}(t_j),\xi(t_n)  \bigr)_{{\mathbb L}^2}\Bigr] \\
&&\qquad {  =} \tau \sum_{j=0}^{N-1} \me\Bigl[\Bigl( {\Pi_h^1U_{h\t}(t_j)},  \me\bigl[\tau \sum_{n=j+1}^NA_0^{n-j}\xi(t_n) \big|\mf_{t_j} \bigr]  \Bigr)_{{\mathbb L}^2}\Bigr] \, .
\end{eqnarray*}
{Since the second argument is ${\mathbb V}^1_h$-valued,} we may skip the projection operator in the first argument, and may continue instead
\begin{eqnarray*}
&&\qquad = \tau \sum_{j=0}^{N-1} \me\Bigl[ \Bigl( {U_{h\t}(t_j)}, { \Pi_h^0}\me \bigl[\tau \sum_{k=j+1}^NA_0^{k-j}\xi(t_k) \big|\mf_{t_j} \bigr]\Bigr)_{{\mathbb L}^2}  \Bigr] \, .\qquad \qquad \quad  
\end{eqnarray*}
Because of \eqref{w1003e15a}, the latter equals
$$=\tau \sum_{j=0}^{N-1} \me \Bigl[ \bigl( U_{h\tau} (t_j), -\Pi_h^0Y_0(t_j) 
\bigr)_{{\mathbb L}^2}\Bigr]\,,\qquad \qquad \qquad \qquad \qquad \quad $$
which is the first part of the claim. 

The remaining part can be deduced from the fact that $Y_1(t_j)=-\me\bigl[ A_0^{N-j}\eta\big|\mf_{t_j}\bigr]$ for $j=0,1,\cds,N-1$,  and the following calculation:
\bel{w1003e22}
\bal
\me\bigl[ \big(\h LU_{h\t},\eta\big)_{{\mathbb L}^2}\bigr]
=&\me \Bigl[ \Bigl( \tau \sum_{j=0}^{N-1}A_0^{N-j}\Pi_h^1U_{h\t}(t_j) ,\eta \Bigr)_{{\mathbb L}^2}\Bigr]\\
=&\t \sum_{j=0}^{N-1}\me\Bigl[\Bigl( \Pi_h^1U_{h\t}(t_j),\me \bigl[ A_0^{N-j}\eta \big|\mf_{t_j}\bigr]\Bigr)_{{\mathbb L}^2}\Bigr] \\
=&\tau \sum_{j=0}^{N-1}\me \Bigl[\Bigl( U_{h\t}(t_j), -\Pi_h^0Y_1(t_j)\Bigr)_{{\mathbb L}^2}\Bigr]\qquad
\forall\, {U_{h\tau} \in {\mathbb U}_{h\tau}}\, .
\eal
\ee

\no{\bf Step 2.} By \eqref{rep1} and \eqref{w1003e14}, we can rewrite $\cJ_{\t}(X_{h\t} ,U_{h\t} )$ as follows:
\begin{eqnarray*}
&&\cJ_{\t}(X_{h\t} ,U_{h\t} )\\
&&\quad =\frac 1 2\lt[ \|X_{h\t} - \Pi_\t{\wt X} \|_{L^2_{{\mathbb F}}(\Omega; L^2(0,T; {\mathbb L}^2))}^2+\|U_{h\t}\|_{L^2_{{\mathbb F}}(\Omega; L^2(0,T; {\mathbb L}^2))}^2+\a\| {X_{h\t}(T)}-{\wt X}(T)\|_{L^2_{{\mathcal F}_T }(\O;\dbL^2)}^2\rt]\\
&&\quad =\frac 1 2\bigg[ \Bigl( \G {\Pi_h^1 X_0} +LU_{h\t}  +f  -\Pi_\t {\wt X},   \G {\Pi_h^1 X_0}+LU_{h\t} +f  -\Pi_\t{\wt X} \Bigr)_{L^2_{{\mathbb F}}(\Omega; L^2(0,T; {\mathbb L}^2))} \\
&&\qq\qq + ( U_{h\t} ,U_{h\t})_{L^2_{{\mathbb F}}(\Omega; L^2(0,T; {\mathbb L}^2))}\\
 &&\qq\qq
   +\a\Bigl( \h\G {\Pi_h^1 X_0} +\h LU_{h\t}+\h f- {\wt X(T)}, \h\G { \Pi_h^1 X_0}+\h LU_{h\t}+\h f- {\wt X(T)}  \Bigr)_{L^2_{{\mathcal F}_T}(\Omega; {\mathbb L}^2)}\bigg]\, ,
   \end{eqnarray*}
where $\Pi_\t\wt X(t)=\wt X(t_n)$, for $t\in [t_n,t_{n+1})$, $n=0,1,\cds,N-1$.
Rearranging terms then leads to
   \beq
=& \frac 1 2 \bigg[  \Bigl(\bigl[\mathds{1} +L^*L+\a\h L^*\h L \bigr]U_{h\t} ,U_{h\t}  \Bigr)_{L^2_{{\mathbb F}}(\Omega;L^2(0,T;{\mathbb L}^2))}  \\
 &       +2 \Bigl(\bigl[L^*\G+\a\h L^*\h\G \bigr] {\Pi^1_h X_0} +L^*f +\a\h L^*\h f-L^*\Pi_\t{\wt X} -\a\h L^*{\wt X(T)}  ,U_{h\t} \Bigr)_{L^2_{{\mathbb F}}(\Omega;L^2(0,T;{\mathbb L}^2))}\\
   &\q  +\Big\{ \Bigl( \G {\Pi_h^1 X_0} +f -\Pi_\t {\wt X}, \G {\Pi_h^1 X_0} +f -\Pi_\t{\wt X} \Bigr)_{L^2_{\mathbb F}(\Omega;L^2(0,T; {\mathbb L}^2))}\\
  &\qq +\a \Bigl( \h\G {\Pi_h^1 X_0}+\h f-{\wt X}(T)  , \h\G \Pi_h^1X_{0}+\h f-{\wt X}(T)  \Bigr)_{L^2_{{\mathcal F}_T}(\O;\dbL^2)} \Big\}     \bigg]\\
=:& \frac 1 2\Big[ \bigl( NU_{h\t} ,U_{h\t} \bigr)_{L^2_{{\mathbb F}}(\Omega;L^2(0,T;{\mathbb L}^2))} +2 \bigl( H({\Pi_h^1 X_0},f,\wt X),U_{h\t}\bigr)_{L^2_{{\mathbb F}}(\Omega;L^2(0,T;{\mathbb L}^2))}
    + M(\Pi_h^1X_{0},f,\wt X)  \Big]\, .    
\eeq
Since $N  =\mathds{1} +L^*L+\a\h L^*\h L $ is positive {definite}, there exists a unique $U^*_{h\t} \in {\mathbb U}_{h\t}$ such that
\beq
NU^*_{h\t} +H({\Pi_h^1 X_0},f,\wt X) =0\, .
\eeq
Therefore, for any $U_{h\t} \in {\mathbb U}_{h\t}$ such that $U_{h\t} \neq U^*_{h\t}$, 
\begin{eqnarray*}
&&\cJ_{\t}(X_{h\t},U_{h\t})-\cJ_{\t}(X^*_{h\t}, U^*_{h\t})\\
&&\quad = \Bigl( NU^*_{h\t}+H({\Pi_h^1 X_0},f,\wt X), U_{h\t}-U^*_{h\t}\Bigr)_{L^2_{\mathbb F}(\Omega; L^2(0,T; {\mathbb L}^2))}\\
&&\qquad
+\frac 1 2 \Bigl( N(U_{h\t}-U^*_{h\t} ), U_{h\t}-U^*_{h\t}\Bigr)_{L^2_{\mathbb F}(\Omega; L^2(0,T; {\mathbb L}^2))}\\
&&\quad =\frac 1 2 \Bigl( N(U_{h\t}-U^*_{h\t}), U_{h\t}-U^*_{h\t}\Bigr)_{L^2_{\mathbb F}(\Omega; L^2(0,T; {\mathbb L}^2))}\\
&&\quad >0\,,
\end{eqnarray*}
which means that $U^*_{h\t}$ is the unique optimal control, and $(X^*_{h\t},U^*_{h\t})$ is the unique optimal pair.

\no{\bf Step 3.} By the definition of $N, H,L^*,\h L^*$,  and properties \eqref{w1003e01} and
 \eqref{rep1}, we can get
\beq
0=&N U^*_{h\t} +H({\Pi_h^1 X_0},f,\wt X)\\
%=&(\mathds{1} +L^*L+\a\h L^*\h L )U^*_{h\t}(\cd)+ \lt(L^*\G+\a\h L^*\h\G \rt)X_{0h}(\cd)+L^*f(\cd)+\a\h L^*\h f-L^*\Pi_h^1\wt X(\cd)-\a\h L^*\Pi_h^1\wt X(T)\\
=&U^*_{h\t} +L^*\lt(\G {\Pi_h^1 X_0} +L U^*_{h\t} +f -{\wt X}\rt)+\a\h L^*\lt(\h\G {\Pi_h^1 X_0} +\h L U^*_{h\t} +\h f-{\wt X}(T)\rt)\\
=&U^*_{h\t} -\Pi_h^0\lt[Y_0\bigl(\cd;X^*_{h\t} -{\wt X} \bigr)+Y_1\lt(\cd;\a\bigl(X^*_{h\t}(T)- {\wt X}(T)\bigr)\rt)\rt]\\
=&U^*_{h\t} -\Pi_h^0 Y_{h\t}\,,\\
\eeq
which is \eqref{w1003e12}. This completes the proof.
\end{proof}

We are now ready to verify strong rates of convergence for the solution to {\bf SLQ$_{h\t}$}; it is as in Section \ref{rate-1} that the reduced cost functional
$\h\cJ_{h\t}: {\mathbb U}_{h\t} \rightarrow {\mathbb R}$ is used, which is defined via
\beq
\h\cJ_{h\t}(U_{h\t})=\cJ_{\t} \bigl(\cS_{h\t}(U_{h\t}),U_{h\t}\bigr)\, ,
\eeq
where $\cS_{h\t}: {\mathbb U}_{h\tau} \rightarrow {\mathbb X}_{h\tau}$ is the solution operator to the forward equation \eqref{w1212e3}$_1$. Moreover, we use
the solution operator $\cT_{h\t}: {\mathbb X}_{h\tau} \rightarrow {\mathbb X}_{h\tau}$ for the first solution component of the backward equation \eqref{w1212e3}$_2$.
%and $\cT_{h\t}: $ (first component to the backward equation) and the reduced cost functional $\h\cJ_{h\t}(\cd)$ on $\cU_{h\t}$ by 
%\beq
%\h\cJ_{h\t}(U_{h\t}(\cd))=\cJ_{h\t}\lt(\cS_{h\t}(U_{h\t}(\cd)),U_{h\t}(\cd)\rt).
%\eeq

\bt{rate2}
Suppose that $X(0)\in \dbH_0^1\cap\dbH^2$, and
\bel{w1212e2}
\sum_{n=0}^{N-1}\int_{t_n}^{t_{n+1}}\me\big[\lt\|\si(t)-\si(t_n)\rt\|_{\dbH_0^1}^2\big]+\lt\|\wt X(t)-\wt X(t_n)\rt\|_{\dbL^2}^2
+\lt\|\wt X(t)-\wt X(t_{n+1})\rt\|_{\dbL^2}^2 \rd t\leq C\t\, .
\ee
Let $(X_h^*,Y_h, Z_h, U^*_h)$ be the solution to problem {\bf SLQ$_h$}, and $(X^*_{h\t},Y_{h\t}, U^*_{h\t})$ be the solution to problem {\bf SLQ$_{h\t}$}. There exists $C \equiv C(X_0, T)>0$ such that
\begin{eqnarray*}
{\rm (i)} && \sum_{k=0}^{N-1}\me \Bigl[\int_{t_k}^{t_{k+1}}\lt\|U^*_h(t)-U^*_{h\t}(t_k)\rt\|_{\dbL^2}^2\, {\rm d} t\Bigr] \leq C \tau\, ; \\ 
{\rm (ii)}&& \max_{0 \leq k \leq N}  {\mathbb E}\bigl[\|X_h^*(t_k)-X_{h\t}^*(t_k)\|_{\dbL^2}^2\bigr]+ {\mathbb E}\Bigl[\tau \sum_{k=1}^N   \|X_h^*(t_k)-X^*_{h\t}(t_k)\|_{\dbH_0^1}^2\Bigr] \leq C\t\, ;\\ 
{\rm (iii)}&&\max_{0 \leq k \leq N} {\mathbb E}\bigl[\|Y_h(t_k)-Y_{h\t}(t_k)\|_{\dbL^2}^2 \bigr]+ {\mathbb E}\Bigl[\tau \sum_{k=0}^{N-1}  \|Y_h(t_k)-Y_{h\t}(t_k)\|_{\dbH_0^1}^2\Bigr] \leq C\t\, .
\end{eqnarray*}
\et
%%%%

\begin{proof} We divide the proof into three steps.

\no {\bf Step 1.}
We  follow the argumentation in the proof of Theorem \ref{rate1}.
For every $U_{h\t},\,R_{h\t} \in {\mathbb U}_{h\t}$, the first Gateaux derivative $D\h\cJ_{h\t}(U_{h\t})$, and the second Gateaux derivative
$D^2\h\cJ_{h\t}(U_{h\t})$ satisfy
\bel{w1206e3}
\bal
& D\h\cJ_{h\t}\lt(U_{h\t}\rt)=U_{h\t}-\Pi_h^0\cT_{h\t}\bigl(\cS_{h\t}\lt(U_{h\t}\rt)\bigr)\, ,\\
&  {\mathbb E}\bigl[ \bigl( D^2 \widehat{\mathcal J}_{h\t}(U_{h\t}) R_{h\t},R_{h\t} \bigr)_{L^2(0,T; {\mathbb L}^2)} \bigr]\geq {\mathbb E}\bigl[\Vert R_{h\t}\Vert^2_{L^2(0,T; {\mathbb L}^2)}\bigr]\, .
\eal
\ee
Define the (piecewise constant) operator ${\Pi_\t}: L^2_\dbF\bigl(\O;C(0,T; \dbV_h^0)\bigr)\rightarrow {\mathbb U}_{h\t}$ by
$$ \Pi_\t U_h(t):=U_h(t_n)  \qquad \forall\, t \in [t_n,t_{n+1}) \qquad n=0,1,\cds, N-1\, .$$ 
By putting $R_{h\t} = U^*_{h\t} - \Pi_\t U^*_h$ in \eqref{w1206e3}, and applying the fact
$D\widehat{\mathcal J}_{h\t}(U^*_{h\t})=D\h\cJ_h(U^*_h) = 0$, we see that
\begin{eqnarray}\nonumber
&& {\mathbb E}\bigl[\lt\| U^*_{h\t} - {\Pi_\t U^*_h}\rt\|_{L^2(0,T; {\mathbb L}^2)}^2\bigr]\\ \nonumber
&&\quad 
\leq  {\mathbb E}\Bigl[ \bigl( D\widehat{\mathcal J}_{h\t}(U^*_{h\t}) ,U^*_{h\t} - \Pi_\t U^*_h\bigr)_{{L^2(0,T; {\mathbb L}^2)}} - 
\bigl( D\widehat{\mathcal J}_{h\t}(\Pi_\t U^*_h) ,U^*_{h\t} - \Pi_\t U^*_h\bigr)_{{L^2(0,T; {\mathbb L}^2)}}\Bigr]\\ \label{w1206e4}
&&\quad = {\mathbb E}\Bigl[ \bigl(  D\widehat{\mathcal J}_h(U^*_h)-  D\widehat{\mathcal J}_{h}(\Pi_\t U^*_h),U^*_{h\t} - \Pi_\t U^*_h\bigr)_{L^2(0,T; {\mathbb L}^2)}\Bigr] \\ \nonumber
&&\qquad + {\mathbb E}\Bigl[\bigl(  D\widehat{\mathcal J}_{h}(\Pi_\t U^*_h)- D\widehat{\mathcal J}_{h\t}(\Pi_\t U^*_h) ,U^*_{h\t} - \Pi_\t U^*_h
\bigr)_{L^2(0,T; {\mathbb L}^2)}\Bigr]\, .\\ \nonumber
%&&\quad  =: I' + II'\, .
%\end{eqnarray}
\end{eqnarray}
Therefore,
\begin{eqnarray}\nonumber
&&{\mathbb E}\bigl[\lt\| U^*_{h\t} - {\Pi_\t U^*_h}\rt\|_{L^2(0,T; {\mathbb L}^2)}^2\bigr]\\ \label{w1212e1}
&&\quad \leq 2 \me \bigl[ \Vert  D\widehat{\mathcal J}_h(U^*_h)-  D\widehat{\mathcal J}_{h}(\Pi_\t U^*_h) \Vert_{L^2(0,T;\dbL^2)}^2\bigr] +2\me\bigl[ \Vert D\widehat{\mathcal J}_{h}(\Pi_\t U^*_h)- D\widehat{\mathcal J}_{h\t}(\Pi_\t U^*_h)
\Vert_{L^2(0,T;\dbL^2)}^2\bigr]\\ \nonumber
&&\quad =: 2I'+2II'\, .
\end{eqnarray}
We use \eqref{derivative-semidisc} and \eqref{pontr1a}, and stability properties of the projection 
$\Pi_h^0$
 to bound $I'$ as follows,
\bel{w1206e5}
\bal
I' =& {\mathbb E}\Bigl[ \lt\| U^*_h - \Pi_\t U^*_h +  \Pi_h^0{\mathcal T}_h \bigl({\mathcal S}_h(\Pi_\t U^*_h) \bigr) - \Pi_h^0{\mathcal T}_h\bigl({\mathcal S}_h(U^*_h) \bigr)\rt\|_{L^2(0,T; {\mathbb L}^2)}^2\Bigr] \\
\leq&2  {\mathbb E}\Bigl[\lt\| U^*_h - \Pi_\t U^*_h\rt\|_{L^2(0,T; {\mathbb L}^2)}^2+\lt\|{\mathcal T}_h \bigl({\mathcal S}_h(\Pi_\t U^*_h) \bigr) -{\mathcal T}_h\bigl({\mathcal S}_h(U^*_h) \bigr) \rt\|_{L^2(0,T; {\mathbb L}^2)}^2\Bigr]\, .
\eal
\ee
 By  stability properties of solutions to {\bf BSPDE}$_h$ \eqref{bsde}, and the discretization \eqref{sde} of 
 {\bf BSPDE}, we obtain
  \begin{eqnarray}\nonumber
 && {\mathbb E}\Bigl[\lt\|{\mathcal T}_h \bigl({\mathcal S}_h(\Pi_\t U^*_h) \bigr) -{\mathcal T}_h\bigl({\mathcal S}_h(U^*_h) \bigr) \rt\|_{L^2(0,T; {\mathbb L}^2)}^2\Bigr]\\ \label{w1206e6}
 &&\quad \leq  C  \me \Bigl[\lt\| \bigl(\cS_h(U^*_h)-\cS_h(\Pi_\t U^*_h)\bigr)(T) \rt\|^2_{\dbL^2}+ \lt\| \cS_h(U^*_h)-\cS_h(\Pi_\t U^*_h) \rt\|^2_{L^2(0,T;\dbL^2)} \Bigr] \\
 \nonumber
 &&\quad \leq   C \lt\| U^*_h-\Pi_\t U^*_h \rt\|^2_{L^2_\dbF(\O;L^2(0,T;\dbL^2))}\, .
 \end{eqnarray}
By the optimality condition \eqref{pontr1a}, estimate \eqref{bspde19}, and Theorem \ref{rate1} (i)
%and due to Example \ref{w1201ex1} 
%may apply Lemma \ref{error_bspde_holder}, (i) for (\ref{bshe1a}) to conclude that\footnote{Dec 10th: Yanqing, I suggest that we work more with the BSPDE-section here -- i.e., the time regularity results we have there: is the above argument correct? A technicality here is that in Lemma \ref{error_bspde_holder}, (i) we ask more for $f$: but is maybe ${\mathbb L}^2$-valued already sufficient there -- and also here we need hoelder continuity in time in ${\mathbb L}^2$ only.}
%
%and the fact $\wt X \in L^2(0,T;\dbH_0^1)$, 
we have
\begin{eqnarray}\nonumber
&& \lt\| U^*_h-\Pi_\t U^*_h \rt\|^2_{L^2_\dbF(\O;L^2(0,T; {\mathbb L}^2))} \\
\label{w1206e7}
&&\quad  \leq C\sum_{k=0}^{N-1}\int_{t_k}^{t_{k+1}}\int_{ t_k}^t \me \Bigl[ \t \bigl\Vert  -\D_h Y_h(s)+ \bigl(X^*_h(s)-\Pi_h^1\wt X(s)  \bigr)\bigr\Vert_{\dbL^2}^2 +\ \Vert Z_h(s) \Vert^2_{\dbL^2}\Bigr]  \rd s \rd t\\ \nonumber
&&\quad  \leq C\t \int_0^T \me \Bigl[ \lt\|-\D_h Y_h(s)+\bigl(X^*_h(s)-\Pi_h^1\wt X(s)  \bigr)\rt\|_{\dbL^2}^2 +\lt\| Z_h(s) \rt\|_{\dbL^2}^2 \Bigr] \rd s \\ \nonumber
&&\quad \leq  C \t\, . 
\end{eqnarray}

Next, we turn to $II'$, for which we use the representations \eqref{derivative-semidisc}, \eqref{w1206e3} and the stability property of $\Pi_h^0$ to conclude
%\beq
%II' \leq& {\mathbb E}\Bigl[\bigl( \Pi_\t D\widehat{\mathcal J}_{h}(\Pi_\t U^*_h)- D\widehat{\mathcal J}_{h\t}(\Pi_\t U^*_h) ,U^*_{h\t} - \Pi_\t U^*_h
%\bigr)_{L^2(0,T; {\mathbb L}^2)}\Bigr] \\
%\leq&
%\lt\|  \Pi_\t {\color{red} \Pi_h^0} {\mathcal T}_h \bigl({\mathcal S}_h(\Pi_\t U^*_h) \bigr) - {\color{red} \Pi_h^0}{\mathcal T}_{h\t} \bigl({\mathcal S}_{h\t}(\Pi_\t U^*_h)  \rt\|_{L^2_{{\mathbb F}}(\Omega; L^2(0,T; {\mathbb L}^2))} \lt\|U^*_{h\t} - \Pi_\t U^*_h\rt\|_{L^2_{{\mathbb F}}(\Omega; L^2(0,T; {\mathbb L}^2))} \\
%\leq &\lt( II'_{a,1}+II'_{a,2}\rt)   \lt\|U^*_{h\t} - \Pi_\t U^*_h\rt\|_{L^2_{{\mathbb F}}(\Omega; L^2(0,T; {\mathbb L}^2))}\,, \\
%\eeq
\beq
II' =&
\me  \bigl[\lt\|   \Pi_h^0 {\mathcal T}_h \bigl({\mathcal S}_h(\Pi_\t U^*_h) \bigr) - 
                \Pi_h^0{\mathcal T}_{h\t} \bigl({\mathcal S}_{h\t}(\Pi_\t U^*_h)\bigr)  \rt\|^2_{ L^2(0,T; {\mathbb L}^2)} \bigr] \\
\leq &2\me \bigl[\lt\| {\mathcal T}_h \bigl({\mathcal S}_{h}(\Pi_\t U^*_h) \bigr) -  {\mathcal T}_h \bigl({\mathcal S}_{h\t}(\Pi_\t U^*_h) \bigr) \rt\|^2_{L^2(0,T; {\mathbb L}^2)} \\
  &+\lt\|  {\mathcal T}_h \bigl({\mathcal S}_{h\t}(\Pi_\t U^*_h) \bigr) -{\mathcal T}_{h\t} \bigl({\mathcal S}_{h\t}(\Pi_\t U^*_h)  \rt\|^2_{L^2(0,T; {\mathbb L}^2)}\bigr]\\ 
=: &2\lt( II'_{a,1}+II'_{a,2}\rt) \,. \\
\eeq
In order to bound $II'_{a,1}$, we use stability properties 
for {\bf SPDE}$_h$ \eqref{sde}, {\bf BSPDE}$_h$ \eqref{bshe1a}, in combination with the error estimate \eqref{euler1} for \eqref{sde} to conclude
$$II'_{a,1} \leq C \me  \bigl[ \lt\Vert\lt( {\mathcal S}_h(\Pi_\t U^*_h) - {\mathcal S}_{h\t}(\Pi_\t U^*_h)\rt)(T)\rt\Vert_{ \dbL^2}^2+\Vert {\mathcal S}_h(\Pi_\t U^*_h) - {\mathcal S}_{h\t}(\Pi_\t U^*_h)\Vert_{L^2(0,T; {\mathbb L}^2)}^2 \bigr]  \leq C\t\, . $$
To bound $II'_{a,2}$, it is easy to see
\beq
\bal
II'_{a,2}\leq&  2\sum_{n=0}^{N-1}\me\lt[\int_{t_n}^{t_{n+1}}\lt\|Y_h(t;\cS_{h\t}(\Pi_\t U^*_h))-Y_h(t_n;\cS_{h\t}(\Pi_\t U^*_h))\rt\|^2_{\dbL^2} \rd t\rt]\\
 &+2T \max_{0\leq n\leq N}\me \bigl[\|Y_h(t_n;\cS_{h\t}(\Pi_\t U^*_h))-Y_h^n(\cS_{h\t}(\Pi_\t U^*_h))\|_{{\mathbb L}^2}^2\bigr].
\eal
\eeq
By \eqref{derivative-semidisc}, Lemma \ref{w0911l2} (i), stable property of \eqref{w1212e3}$_1$, we can get
\beq
&\sum_{n=0}^{N-1}\me\lt[\int_{t_n}^{t_{n+1}}\lt\|Y_h(t;\cS_{h\t}(\Pi_\t U^*_h))-Y_h(t_n;\cS_{h\t}(\Pi_\t U^*_h))\rt\|^2_{\dbL^2} \rd t\rt]\\
&\quad \leq C\t\lt(\me\|\cS_{h\t}(\Pi_\t U^*_h)(T)\|_{\dbH_0^1}^2+\|\wt X(T)\|_{\dbH_0^1}^2+\|\cS_{h\t}(\Pi_\t U^*_h)\|^2_{L^2_{\dbF}(\O;L^2(0,T;\dbL^2))}+\|\wt X\|^2_{L^2(0,T;\dbL^2)}\rt)\\
&\q \leq C\t\lt(\| X_0\|_{\dbH_0^1}^2+\|\wt X(T)\|_{\dbH_0^1}^2+\|\si\|^2_{L^2_{\dbF}(\O;L^2(0,T;\dbH_0^1))}+\|Y_h\|^2_{L^2_{\dbF}(\O;L^2(0,T;\dbH_0^1))}+\|\wt X\|^2_{L^2(0,T;\dbL^2)}\rt)\\
&\quad \leq  C\t\, .
\eeq
Utilizing Theorem \ref{w0911t2} for {\bf BSPDE}$_h$ \eqref{bsde} with $\Pi_hf=\cS_{h\t}(\Pi_\t U^*_h)-\Pi_h^1\wt X$ 
and \eqref{w1212e2}, we can find that
\beq
 &\max_{0\leq n\leq N}\me \bigl[\|Y_h(t_n;\cS_{h\t}(\Pi_\t U^*_h))-Y_h^n(\cS_{h\t}(\Pi_\t U^*_h))\|_{{\mathbb L}^2}^2\bigr] \\
&\quad \leq C  {\mathbb E}\Bigl[
 \sum_{n=0}^{N-1}  \int_{t_n}^{t_{n+1}} \bigl\Vert \nabla \bigl[Y_h(s;\cS_{h\t}(\Pi_\t U^*_h))-Y_h(t_n;\cS_{h\t}(\Pi_\t U^*_h))
\bigr] \bigr\Vert^2_{{\mathbb L}^2} + \Vert \wt X(s) - \wt X(t_n)\Vert^2_{{\mathbb L}^2}\, {\rm d}s \Bigr]\\
&\quad \leq C\t\lt[ \max_{0\leq n\leq N} \me \bigl\Vert \nabla X_{h\t}(t_n;\cS_{h\t}(\Pi_\t U^*_h))
 \bigr\Vert^2_{{\mathbb L}^2}+\me\int_0^T\bigl\Vert \nabla \Pi_h^1\wt X(t)
 \bigr\Vert^2_{{\mathbb L}^2}\rd t\rt]\\
 &\qquad +C\me\Bigl[
 \sum_{n=0}^{N-1}  \int_{t_n}^{t_{n+1}}  \Vert \wt X(s) - \wt X(t_n)\Vert^2_{{\mathbb L}^2}\, {\rm d}s \Bigr]\\
&\quad \leq  C\t\, .
\eeq
Here, we apply the representation of $X_{h\t}$ \eqref{rep1}, the fact $\wt X\in L^2(0,T;\dbH_0^1)$,
and condition \eqref{w1212e2}.

Now we insert above estimates into \eqref{w1212e1} to obtain assertion {\rm (i)}.
%\beq
%\lt\| U^*_{h\t} -  U^*_h\rt\|^2_{L^2_\dbF(\O;L^2(0,T;\dbV_h^0))}
%\leq \lt\| U^*_{h\t} - \Pi_\t U^*_h\rt\|^2_{\cU_{h\t}}+\lt\|  \Pi_\t U^*_h -  U^*_h\rt\|^2_{L^2_\dbF(\O;L^2(0,T;\dbV_h^0))}
%\leq C\t.
%\eeq
%%
%That is  {\rm (i)}.

\ms

{\no{\bf Step 2.}
%\footnote{Dec. 12th: I rewrite the proof of Step 2 and 3. Here, Lemma 3.1 can not be used cause the conditions in Lemma 3.1 is so good that the responding conditions in SLQ problem may not hold. } 
For all $k=0,1,\cds,N$, we define $e_{X}^k=X^*_h(t_k)- X^*_{h\t}(t_k)$. Subtracting \eqref{w1003e7} from \eqref{w1013e1a} leads to
\begin{eqnarray*}
e_{X}^{k+1} - e_X^{k} &=&  \tau \Delta_h e_X^{k+1} +
\tau \Pi_h^1[U_h^*(t_k) - U_{h\tau}^*(t_k)] + \int_{t_k}^{t_{k+1}} \Pi_h^1 [\sigma(s) - \sigma(t_k)]\, {\rm d}W(s)\\
&&
+ \int_{t_k}^{t_{k+1}} \lt(\Delta_h [X^*_h(s) - X^*_h(t_{k+1})] + \Pi_h^1[U^*_h(s) - U^*_h(t_k)]\rt)
\, {\rm d}s\, .
\end{eqnarray*}
Testing with $e_{X}^{k+1}$, and using binomial formula, Poincar\'{e}'s inequality, independence, and absorption lead to
\begin{eqnarray*}
&&\frac{1}{2} {\mathbb E}\bigl[ \Vert e_{X}^{k+1}\Vert^2_{{\mathbb L}^2} - \Vert e_{X}^{k}\Vert^2_{{\mathbb L}^2} + \frac{1}{2} \Vert e_{X}^{k+1} - e_{X}^{k}\Vert^2_{{\mathbb L}^2}\bigr] + \frac{\tau}{2} {\mathbb E}\bigl[ \Vert \nabla e_{X}^{k+1} \Vert^2_{{\mathbb L}^2}\bigr]\\
&&\quad \leq C\tau {\mathbb E}\bigl[ \Vert U_h^*(t_k) - U_{h\tau}^*(t_k)\Vert^2_{{\mathbb L}^2}\bigr] + C {\mathbb E}\Bigl[ \Bigl\Vert \int_{t_k}^{t_{k+1}} \Pi_h^1 [\sigma(s) - \sigma(t_k)]\, {\rm d}W(s)\Bigr\Vert_{{\mathbb L}^2}^2\Bigr] \\
&&\qquad + C {\mathbb E}\Bigl[ \int_{t_k}^{t_{k+1}} \Vert \nabla \bigl[ X^*_h(s) - X^*_h(t_{k+1})\bigr] \Vert^2_{{\mathbb L}^2}\, {\rm d}s\Bigr] + C {\mathbb E}\Bigl[ \int_{t_k}^{t_{k+1}}  \Vert U^*_h(s) - U^*_h(t_k)\Vert^2_{{\mathbb L}^2}\, {\rm d}s\Bigr].
\end{eqnarray*}
By taking the sum over all $0\leq k\leq n$ and $0\leq k\leq N-1$, and noting that $e_X^0=0$, we find that
\beq
&\max_{0\leq n\leq N}\me \bigl[\|e_X^n\|_{\dbL^2}^2 \bigr]+\sum_{n=1}^N\t\me \bigl[\|\nb e_X^n\|_{\dbL^2}^2 \bigr]\\
&\quad \leq C \sum_{k=0}^{N-1}\me\Bigl[\t \|U^*_h(t_k)-U^*_{h\t}(t_k)\|_{\dbL^2}^2
+\int_{t_k}^{t_{k+1}} \Big( \|\si(s)-\si(t_k)\|_{\dbL^2}^2 +\\
&\qquad +\|\nb \lt[X^*_h(s) - X^*_h(t_{k+1})\rt]\|_{\dbL^2}^2
+\|U^*_h(s)-U^*_h(t_k)\|_{\dbL^2}^2\Big)  \rd s \Bigr]\, .
\eeq
By \eqref{w1212e1}, the first term on the right-hand side is bounded by $C \tau$. 
We use It\^{o} isometry for the second term, and H\"older regularity in time of $\sigma$ to bound it equally. Adopting the method in \eqref{w1206e7}, we can bound the third term by
$
C\t \lt( \|\D_h X^*_h(0)\|_{\dbL^2}^2 +\|\nb\Pi_h^1U^*_h(t)\|_{L^2_\dbF(\O;L^2(0,T;\dbL^2))}^2+\|\si(t)\|_{L^2_\dbF(\O;L^2(0,T;\dbH^2))}^2\rt).
$
 We use \eqref{w1206e7} to bound the last term by $C\tau$.
That is assertion (ii).
%\beq
%e^k_{X}^\t =X^*_h(t_k)- X^*_{h\t}(t_k) .   
%\,e_{Y}^\t(k)=Y_h(t_k)- Y_{h\t}(t_k).
%As \eqref{w1208e1}, we can have
%\bel{w1004e4}
%\bal
%&\me\lt\|e_{X}^\t (k+1)\rt\|_{\dbL^2}^2+\me\lt\| \nb e_{ X}^\t (k+1)\rt\|_{\dbL^2}^2 \t\\
%\leq&(1+2\t)\me \|e_{X}^\t (k)\|_{\dbL^2}^2+2\me\int_{t_k}^{t_{k+1}} \lt[\lt\|\nb( X^*_h(t)- X^*_h(t_{k+1}))\rt\|_{\dbL^2}^2+\lt\|  U^*_h(t)-U^*_{h\t}(t_k) \rt\|_{\dbL^2}^2\rt]  \rd t.\\
%\eal
%\ee
%Since the right side of \eqref{w1004e4} is bounded by
%\beq
%(1+2|\pi|)^{N}\me \|e_{ X}^\t (0)\|^2
%+\sum_{i=0}^k(1+2\t)^{N}\me\int_{t_i}^{t_{i+1}}\lt[\lt\|\nb( X^*_h(t)- X^*_h(t_{k+1}))\rt\|_{\dbL^2}^2+\lt\|  U^*_h(t)-U^*_{h\t}(t_k) \rt\|_{\dbL^2}^2\rt] \rd t,
%\eeq
%which is bounded by $C\t$, by the fact $e_{ X}^\t(0)=0$, regularity of $X^*_h$ as \eqref{w1208e2} and estimate {\rm (i)}.
%Hence, we can get
%\beq
%\max_{0\leq k\leq N}\me\lt\|e_{X}^\t (k)\rt\|_{\dbL^2}^2\leq C\t.
%\eeq
%Then, summing up \eqref{w1004e4} from $k=0$ to $N-1$, we find that
%\beq
%\sum_{k=1}^{N} \me\lt\| \nb e_{ X}^\t (k)\rt\|_{\dbL^2}^2 \t\leq C\t.
%\eeq
%So the first estimate of (ii) is proved.

\ms

\no{\bf Step 3.} 
Firstly, we introduce an auxiliary BSDE
\bel{w1212e5}
\lt\{
\bal
&Y_{\t}(t_{n+1})-Y_{\t}(t_n)= \tau \lt[-\D_hY_{\t}(t_{n})+\lt(X^*_{h\t}(t_{n+1})-\Pi_h^1\wt X(t_{n+1})\rt) \rt]\\
&\qq\qq\qq\qq\qq+\int_{t_n}^{t_{n+1}}Z_{\t}(t)\, {\rm d}W(t) \q n=0,1,\cds,N-1\, ,\\
& Y_{\t}(T)=-\a\lt(X^*_{h\t}(T)-\Pi_h^1\wt X(T)\rt).
\eal
\rt.
\ee
It is easy to see that $Y_\t=Y_{h\t}$. Define $e_Y^n=Y_{h}(t_n)-Y_\t(t_n)$, $n=0,1,\cds,N$. 
With the same argument as that in the proof of Theorem \ref{w0911t2}, we can deduce
\beq
&\max_{0\leq n\leq N}\me \bigl[\|e_Y^n\|_{\dbL^2}^2\bigr]+\sum_{n=1}^N\t\me  \bigl[\lt\|\nb e_Y^n\rt\|_{{\mathbb L}^2}^2 \bigr]\\
\leq& C\sum_{k=0}^{N-1}  
\int_{t_k}^{t_{k+1}}{\mathbb E} \Big[ \Vert \nb \bigl[Y_h(s) - Y_h(t_{k})\bigr]\|_{\dbL^2}^2+\|X^*_h(s) - X^*_{h\t}(t_{k+1})\|_{\dbL^2}^2
+\|\wt X(s)-\wt X(t_{k+1})\|_{\dbL^2}^2\Big] \rd s \, .
\eeq
Applying Lemma \ref{w0911l2} (ii), the first integral term is bounded by 
\beq
C\t\lt\{ \|\D_h X_h(0) \|^2_{\dbL^2}+\|\D_h \Pi_h^1 \wt X(T) \|^2_{\dbL^2}+\int_0^T\me\bigl[  \|\nb X^*_h(t)\|_{\dbL^2}^2+ \|\nb\Pi_h^1 \wt X^*_h(t)\|_{\dbL^2}^2+ \| \Pi_h^1U^*_h(t)\|_{\dbL^2}^2 \bigr] \rd t\rt\}\, .
\eeq
It remains to estimate the second integral term,  which is bounded by
\beq
&C\sum_{k=0}^{N-1}  
\int_{t_k}^{t_{k+1}}\me\Big[ \|X^*_h(s) - X^*_{h}(t_{k+1})\|_{\dbL^2}^2+\|X^*_h(t_{k+1}) - X^*_{h\t}(t_{k+1})\|_{\dbL^2}^2
\Big] \rd s  \\
&\quad \leq C\t\lt\{ \|\nb X_h(0) \|^2_{\dbL^2} +\int_0^T\me \bigl[  \| \Pi_h^1U^*_h(t)\|_{\dbL^2}^2 +\|\nb \Pi_h^1\si(t)\|_{\dbL^2}^2 \bigr]\rd t\rt\}\\
&\qquad +C \max_{0 \leq k \leq N}  {\mathbb E}\bigl[\|X_h^*(t_k)-X_{h\t}^*(t_k)\|_{\dbL^2}^2\bigr]\\
&\quad \leq C\t\, .
\eeq
Assertion (iii) now follows from the above three statements and conditions on $X_0,\si,\wt X$. }
\end{proof}

\section{The gradient descent method to solve {\bf SLQ}$_{h\t}$}\label{numopt}

By Theorem \ref{MP}, solving minimization problem {\bf SLQ$_{h\t}$} is 
equivalent to solving the system of coupled forward-backward difference equations 
\eqref{w1212e3} and \eqref{w1003e12}. We may exploit the variational
character of problem {\bf SLQ$_{h\t}$} to construct a gradient descent method
 {\bf SLQ$_{h\t}^{\rm grad}$}
where approximate iterates of the optimal control $U^*_{h\t}$ in the 
Hilbert space $\dbU_{h\t}$ are obtained; see also \cite{Nesterov04,Kabanikhin12}.
%
%
%That system of equations may not be decoupled directly, 
%so that we could adopt an
%iterative strategy. For general forward-backward coupled equations, it seems that there does not exist a workable and
% efficient iterative scheme to solve them. However, for the problem {\bf SLQ$_{h\t}$}, we can 
%apply a gradient descent method to iteratively compute approximations of the optimal control $U^*_{h\t}$ in the 
%Hilbert space $\dbU_{h\t}$.
%In this section,
%we only list a basic method; for further variants of gradient descent methods, we refer the reader to \cite{Nesterov04,Kabanikhin12}.

\ms

\begin{algorithm}\label{alg1}
{\bf ({\bf SLQ$_{h\t}^{\rm grad}$})}
Let $U_{h\t}^{(0)}\in \dbU_{h\t}$, and fix $\kappa > 0$. For any $\ell \in {\mathbb N}_0$, update $U_{h\t}^{(\ell)} \in {\mathbb U}_{h\tau}$ as follows:
\begin{enumerate}
\item[1.] Compute $X_{h\t}^{(\ell)}\in \dbX_{h\t}$ by 
%\bel{w0115e4}
\begin{equation*}
\left\{
\bal
& [\mathds{1} - \tau \Delta_h]X^{(\ell)}_{h\t}(t_{n+1})= X^{(\ell)}_{h\t}(t_n)+ \tau \Pi_h^1U^{(\ell)}_{h\t}(t_n) +\Pi_h^1\si(t_n) \D_{n+1}W \q n=0,1,\cds,N-1\, ,\\
& X_{h\t}^{(\ell)}(0)=\Pi_h^1X_0\, .
\eal
\right.
\end{equation*}
\item[2.] Use $X_{h\t}^{(\ell)}\in \dbX_{h\t}$ to compute $Y_{h\t}^{(\ell)}\in \dbX_{h\t}$ via
%bel{w0115e3}
\begin{equation*}
\lt\{
\bal
&[\mathds{1} - \tau \Delta_h]Y^{(\ell)}_{h\t}(t_n) = {\mathbb E}\lt[ Y^{(\ell)}_{h\t}(t_{n+1})- {\tau} \bigl(X^{(\ell)}_{h\t}(t_{n+1})-\Pi_h^1\wt X(t_{n+1})\bigr)\bigl\vert
{\mathcal F}_{t_n}\rt] \\
&\qq\qq\qq\qq\qq\qq\q n=0,1,\cds,N-1\, ,\\
&Y_{h\t}^{(\ell)}(T)=-\a \bigl(X_{h\t}^{(\ell)}(T)-\Pi_h^1\wt X(T) \bigr)\, .
\eal
\rt.
\end{equation*}

\item[3.] Compute the update $U_{h\t}^{(\ell+1)} \in {\mathbb U}_{h\tau}$ via
\beq
U^{(\ell+1)}_{h\t}=U^{(\ell)}_{h\t}-\frac 1 {\kappa} \lt(U^{(\ell)}_{h\t} -\Pi_h^0Y_{h\t}^{(\ell)} \rt) \, .
\eeq
%with $K_0=1+\a T+T^2$.
\end{enumerate}
\end{algorithm}
Note that Steps 1 and 2 are now decoupled: the first step requires to solve a space-time discretization (\ref{esti-time1}) of  {\bf SPDE} (\ref{forw1}), while the second requires to solve the space-time discretization (\ref{w0911e10})$_1$ of the {\bf BSPDE} (\ref{w1205e3})$_2$.
We refer to related works on how to approximate conditional expectations
 \cite{Bouchard-Touzi04, Gobet-Lemor-Warin05, Bender-Denk07, Wang-Zhang11}; 
a similar method to {\bf SLQ$_{h\t}^{\rm grad}$} to solve problem {\bf SLQ$_{h\t}$} 
has been proposed in \cite{Dunst-Prohl16}.

We want to show convergence of {\bf SLQ$_{h\t}^{\rm grad}$} for $\kappa >0$ sufficiently large and $\ell \uparrow \infty$. For this purpose, we recall the notations ${\mathcal S}_{h\tau}, {\mathcal T}_{h\tau}, \widehat{\mathcal J}_{h\tau}$ introduced in Section \ref{rate-2}. For this purpose, we first recall Lipschitz continuity of 
$D\widehat{\mathcal J}_{h\tau}$:
%
%The following is the {\bf gradient descent method SLQ$_{h\t}^{grad}$} subject to Problem {\bf SLQ}$_{h\t}$:\\
%{\bf Step (1)}: Choose $U_{h\t}^0\in \dbU_{h\t}$.\\
%{\bf Step (2)$_a$}: For obtained $U_{h\t}^k$, compute $X_{h\t}^k\in \dbX_{h\t}$ by 
%\bel{w0115e4}
%\left\{
%\bal
%& [\mathds{1} - \tau \Delta_h]X^k_{h\t}(t_{n+1})= X^k_{h\t}(t_n)+ \tau \Pi_h^1U^k_{h\t}(t_n) +\Pi_h^1\si(t_n) \D_{n+1}W \q n=0,1,\cds,N-1,\\
%& X_{h\t}^k(0)=\Pi_h^1X_0.
%\eal
%\right.
%\ee
%{\bf Step (2)$_b$}: For obtained $X_{h\t}^k\in \dbX_{h\t}$, get $Y_{h\t}^k\in \dbX_{h\t}$ via
%\bel{w0115e3}
%\lt\{
%\bal
%&[\mathds{1} - \tau \Delta_h]Y^k_{h\t}(t_n) = {\mathbb E}\lt[ Y^k_{h\t}(t_{n+1})- {\tau} \bigl(X^k_{h\t}(t_{n+1})-\Pi_h^1\wt X(t_{n+1})\bigr)\bigl\vert
%{\mathcal F}_{t_n}\rt] \\
%&\qq\qq\qq\qq\qq\qq\q n=0,1,\cds,N-1,\\
%&Y_{h\t}^k(T)=-\a(X_{h\t}^k(T)-\Pi_h^1\wt X(T)).
%\eal
%\rt.
%\ee
%{\bf Step (2)$_c$}: Get the updated $U_{h\t}^{k+1}$ via
%\beq
%U^{k+1}_{h\t}=U^k_{h\t}-\frac 1 {K_0} \lt(U_{h\t}^k-\Pi_h^0Y_{h\t}^k \rt),\q k=1,2\cds,
%\eeq
%with $K_0=1+\a T+T^2$.
%
since 
$$D^2\h\cJ_{h\t}(U_{h\t})= \lt(\mathds{1} +L^*L+\a\h L^*\h L\rt)U_{h\t}\,,$$
where operators $L,\,\h L$ are defined in \eqref{rep1}, we find $K:=\|\mathds{1} +L^*L+\a\h L^*\h L\|_{\cL(\dbU_{h\t};\dbU_{h\t})}$, such that
\beq
\|D\h\cJ_{h\t}\lt(U^1_{h\t}\rt)-D\h\cJ_{h\t}\lt(U^2_{h\t}\rt)\|_{\dbU_{h\t}}\leq K \lt\|U^1_{h\t}-U^2_{h\t}\rt\|_{\dbU_{h\t}}\, .
\eeq
Indeed,
noting that $\|(\mathds{1}-\t\D_h)^{-1}\|_{\cL(\dbV_h^1;\dbV_h^1)}\leq 1$,
we conclude
\begin{eqnarray*}
\|LU_{h\t}\|_{\dbX_{h\t}}^2&=&\sum_{n=1}^N\tau \me \bigl[\|LU_{h\t}(t_n)\|_{\dbL^2}^2\bigr]
=\sum_{n=1}^N \tau \me\Bigl[ \| \t \sum_{j=0}^{n-1} \lt[(\mathds{1}-\t\D_h)^{-1}\rt]^{n-j}\Pi_h^1U_{h\t}(t_j) \|_{\dbL^2}^2 \Bigr] \\
&\leq&  T^2\|U_{h\t}\|_{\dbU_{h\t}}^2\, ,
\end{eqnarray*}
and
\beq
\|\h LU_{h\t}\|_{L^2_{\mf_T}(\O;\dbL^2)}^2=\me \Bigl[\big\| \t \sum_{j=0}^{N-1} \lt[(\mathds{1}-\t\D_h)^{-1}\rt]^{N-j}\Pi_h^1U_{h\t}(t_j)\big\|_{\dbL^2}^2 \Bigr]\leq  T\|U_{h\t}\|_{\dbU_{h\t}}^2\, .
\eeq
Hence
\beq
K=\|\mathds{1} +L^*L+\a\h L^*\h L\|_{\cL(\dbU_{h\t};\dbU_{h\t})}\leq 1+\a T+T^2.
\eeq
Since {\bf SLQ$_{h\t}^{\rm grad}$} is the gradient descent method for {\bf SLQ$_{h\t}$}, we have the following result.
%
%By standard gradient descent method in Hilbert space, we know that for any $K_0\geq K$, gradient descent 
%method works.
%\er
%
%
%\br{w1220r1}
%Note that \eqref{w0115e4} and \eqref{w0115e3} are both decoupled, which are just Euler method. 
%Since \eqref{w0115e3} is backward, 
%conditional expectation appears to ensure solution's adapted property. There are some works on computation of
%conditional expectations subject to BSDEs, such as \cite{Bouchard-Touzi04, Gobet-Lemor-Warin05, Bender-Denk07, Wang-Zhang11}, 
%and so on.
%A similar method as {\bf SLQ$_{h\t}^{grad}$} to solve problem {\bf SLQ$_{h\t}$} is proposed in \cite{Dunst-Prohl16}.
%\er
%
%By the standard error estimates for gradient method, we can get the following convergence rate for our problem.
%
\bt{gradient-rate} Suppose that $\kappa \geq K$.
Let $\ds \{U^{(\ell)}_{h\t}\}_{\ell \in {\mathbb N}_0} \subset {\mathbb U}_{h\tau}$ be  generated by {\bf SLQ$_{h\t}^{\rm grad}$}, and $U^*_{h\t}$ solve {\bf SLQ$_{h\t}$}. Then
%\bel{w1218e3}
\beq
\lt\{
\bal
&\h\cJ_{h\t}(U^{(\ell)}_{h\t})-\h\cJ_{h\t}(U^*_{h\t})\leq \frac{2 \kappa \Vert U^{(0)}_{h\t}-U^*_{h\t} \Vert_{\dbU_{h\t}}^2}{\ell},\\
& \| U^{(\ell)}_{h\t}-U^*_{h\t} \|_{\dbU_{h\t}}^2\leq \lt(1-\frac 1 {\kappa}\rt)^{\ell}\| U^{(0)}_{h\t}-U^*_{h\t} \|_{\dbU_{h\t}}^2  \qquad  \ell = 1, 2, \cdots .
\eal
\rt.
\eeq
\et

\begin{proof}
We know that $D\h\cJ_{h\t}$ is Lipschitz continuous with constant $K>0$. Also, $\h\cJ_{h\t}$ is strongly convex. 
Hence, the gradient descent method in abstract form is the following iteration (see Algorithm \ref{alg1}, Step 3.)
\bel{w0115e2}
U^{(\ell+1)}_{h\t}=U^{(\ell)}_{h\t}-\frac 1 {\kappa} D\h\cJ_{h\t}(U^{(\ell)}_{h\t}),\q \ell=0,1,2\cds\, .
\ee
By the proof of Theorem \ref{MP}, we have obtained the following facts:
\beq
\lt\{
\bal
&D\h\cJ_{h\t}(U^{(\ell)}_{h\t})=U^{(\ell)}_{h\t}-\Pi_h^0\cT_{h\t}\bigl(\cS_{h\t} (U^{(\ell)}_{h\t})\bigr)\, ,\\
%&D^2\h\cJ_{h\t}(U_{h\t})= \lt(\mathds{1} +L^*L+\a\h L^*\h L\rt)U_{h\t},\\
& \Pi_h^0 \cT_{h\t}\bigl(\cS_{h\t}(U^{(\ell)}_{h\t})\bigr)=-L^*\lt(\G {\Pi_h^1 X_0} +L U^{(\ell)}_{h\t} +f -{\wt X}\rt)\\
 &\qq\qq\qq\qq\qq-\a\h L^*\lt(\h\G {\Pi_h^1 X_0} +\h L U^{(\ell)}_{h\t} +\h f-{\wt X}(T)\rt),
\eal
\rt.
\eeq
where $L,\h L,\G,\h\G,\h f$ are defined in \eqref{rep1} and \eqref{w1003e14}. Via \eqref{w1003e01}, we have 
that $ \Pi_h^0 \cT_{h\t}\bigl(\cS_{h\t}(U^{(\ell)}_{h\t})\bigr)$ is just $Y_{h\t}^{(\ell)}$, the solution of Step 2 in Algorithm \ref{alg1}.
Therefore, \eqref{w0115e2} is consistent with the gradient descent method {\bf SLQ$_{h\t}^{\rm grad}$}. 
The desired error estimates now follow by standard estimates for the gradient descent method (see, e.g. \cite[Theorem 1.2.4]{Nesterov04}).
\end{proof}

\section*{Acknowledgement}

This work was carried out when Yanqing Wang visited the University of T\"ubingen in 2019--2020, supported by a 
DAAD-K.C.~Wong Postdoctoral Fellowship.

%\renewcommand\refname{References}
%\bibliographystyle{siam}         %{abbrv}%{plain}
%             %{unsrt}%{alpha}%{ieeetr}%{acm}
%\bibliography{YQreference}

\end{document}